\documentclass[a4paper,10pt]{article}

\usepackage[T1]{fontenc} 
\usepackage{amsmath,amsthm}
\usepackage[dvips]{graphicx}
\usepackage{amsfonts,amssymb}
\usepackage{geometry}
\usepackage{hyperref}
\usepackage{stmaryrd}
\usepackage{fancyhdr}
\usepackage{url}

\newtheorem{prop}{Proposition}[section]
\newtheorem{LM}{Lemma}[section]
\newtheorem{thm}{Theorem}[section]
\newtheorem{df}{Definition}[section]
\newtheorem{df-prop}{Definition-Proposition}[section]
\newtheorem{cor}{Corollary}[section]
\newtheorem{ass}{Assumption}[section]
\newtheorem{conj}{Conjecture}[section]
\newtheorem*{claim}{Claim}
\newtheorem*{conj*}{Conjecture \ref{distgen}}
\newtheorem*{thme*}{Corollary \ref{asaigen}}
\newtheorem*{thme**}{Theorem \ref{asaiprinc}}
\newtheorem*{thme***}{Theorem \ref{equiv}}

\newtheoremstyle{pourlesremarques}{\topsep}{\topsep}{\normalfont}{}{\bfseries}{.}{ }{}
\theoremstyle{pourlesremarques}
\newtheorem{rem}{Remark}[section]
\def\adots{\mathinner{\mkern2mu\raise 1pt\hbox{.}\mkern 3mu\raise
4pt\hbox{.}\mkern1mu\raise 7pt\hbox{{.}}}}

\title {Conjectures about distinction and Asai $L$-functions of generic representations of general linear groups over local fields}
\author{Nadir Matringe}

\begin{document}
 \maketitle

\begin{abstract}
 Let $K/F$ be a quadratic extension of p-adic fields. The Bernstein-Zelevinsky's classification asserts that generic representations are parabolically induced from quasi-square-integrable representations. We show, following a method developed by Cogdell and Piatetski-Shapiro, that expected equality of the Rankin-Selberg type Asai L-function of a generic representation of $GL(n,K)$, and the Asai $L$-function of its Langlands parameter is equivalent to the truth of a conjecture about classification of distinguished generic representations in terms of the inducing quasi-square-integrable representations. As the conjecture is true for principal series representations, this gives the expression of the Asai L-function of such representations.

\end{abstract}

\section*{Introduction}

Given $K/F$ a quadratic extension of $p$-adic fields, we denote by $\sigma$ the non trivial element of the Galois group of $K$ over $F$. This work concerns itself with two interrelated themes, the first being the properties of the Rankin-Selberg type Asai $L$-functions of generic representations of the group $GL(n,K)$, the second being to know how to recognize when such a representation admits on its space a nonzero $GL(n,F)$-invariant linear form, i.e. when it is distinguished.\\

The local Asai $L$-function of an irreducible representation $\pi$ of $GL(n,K)$, denoted by $L_F^K(\pi)$, defined as the gcd of functions obtained as meromorphic extension of Rankin-Seberg integrals, appears implicitly in Flicker's paper \cite{F1}. Its basic properties and its functional equation are established in \cite{F4}, in the manner that this latter is obtained for $L$-functions of pairs by Jacquet, Piatetski-Shapiro and Shalika in \cite{JPS}. The study of its poles is related with the distinction of the representation $\pi$.\\
There are two other ways to associate an Asai $L$-function to a representation $\pi$ of the group $GL(n,K)$.\\
 The first is by considering the $n$-dimensional representation $\rho$ of the Weil-Deligne $W'_K$ of $K$, associated to $\pi$ by local Langlands correspondence. One then defines by multiplicative induction a representation of the Weil-Deligne group $W'_F$ (which contains $W'_K$ as a subgroup of index $2$), of dimension $n^2$, denoted by $M_{W'_K}^{W'_F}(\rho)$. The Asai $L$-function corresponding to $\pi$, and denoted by $L_F^{K,W}(\pi,s)$ is by definition the classical $L$-function of the representation $M_{W'_K}^{W'_F}(\rho)$, which we denote by $L_F^K(\rho)$.\\
 The second, called the Langlands-Shahidi method, is introduced in \cite{Sh}. We denote by $L_F^{K,U}(\pi)$ the meromorphic function obtained by this process, and the study of its poles is this time related to the fact of knowing when a representation $\pi$ is obtained by base change lift from a unitary group. This work is done in \cite{Go}. We will not concern ourselves with this aspect of distinguished representations here.\\
It is conjectured that these three functions are actually the same (cf. \cite{He}, \cite{K}, \cite{AR}). Henniart proved in \cite{He} that the functions $L_F^{K,U}$ and $L_F^{K,W}$ are equal. Anandavardhanan and Rajan proved in \cite{AR} that the functions $L_F^K$ and $L_F^{K,U}$ coincide for the representations of the discrete series of $GL(n,K)$. These two equalities have global proofs. In this work, as an application of these equalities, we will give in Proposition \ref{asaidiscrete} an explicit computation of the Asai L-function of essentially discrete series representations of the group $GL(n,K)$. \\

We recall that a representation $\pi$ of $GL(n,K)$ is called distinguished if there is a nonzero $GL(n,F)$-invariant linear form on its space (cf. Definition \ref{dfdist}, where we define more generally $\mu$-distinction for a character $\mu$ of $F^*$). In the local context which we are interested in, the first link between distinction and the Asai $L$-function $L_F^K$ is brought to light by Kable in \cite{K}. He shows that if the Asai $L$-function $L_F^K(\pi)$ of a discrete seris representation $\pi$ admits a pole at zero, then the representation $\pi$ is distinguished. Then Anandanvardhanan, Kable and Tandon
show in \cite{AKT} that the Asai $L$-function of a distinguished tempered representation admits a pole at zero. Eventually, it is proved in \cite{M2} that a generic representation of $GL(n,K)$ is distinguished if and only if its Asai $L$-function admits an exceptional pole (in the terminology of Cogdell and Piatetski-Shapiro) at zero.\\

 In this paper, we adapt the method developed by Cogdell and Piatetski-Shapiro in \cite{CP}, for computing $L$-functions of pairs of generic representations, to the framework of Asai $L$-functions.\\
In \cite{CP}, the computation of the function $L(\pi\times \pi',s)$ for two generic representations $\pi$ and $\pi'$ of the group $GL(n,K)$ is boiled down to understand when the representation $\pi \otimes \pi'$ of the product group $GL(n,K)\times GL(n,K)$ is distinguished by the subgroup $GL(n,K)$ diagonally embedded. In this case, this comes down to say that the smooth dual $\pi^{\vee}$ of the representation $\pi$, is isomorphic to the representation $\pi'$. Hence one obtains a simple necessary and sufficient condition on the quasi-square-integrable representations which, parabolically induced, give birth to the representations $\pi$ and $\pi'$. In our case, One doesn't know a necessary and sufficient condition on the inducing data, characterizing distinction, but the following statement is conjectured.

\begin{conj*}
Let $\pi$ be a generic representation of the group $GL(n,K)$, obtained by normalised parabolic induction of quasi-square-integrable representations $\Delta_1, \dots ,\Delta_t$. It is distinguished if and only if there exists a reordering of the ${\Delta _i}$'s, and an integer $r$ between $1$ and $t/2$, such that we have $\Delta_{i+1}^{\sigma} = \Delta_i^{\vee} $ for $i=1,3,..,2r-1$, and $\Delta_{i}$ is distinguished for $i > 2r$.
\end{conj*}

 One then shows in the corollary of Theorem \ref{asailang}, that up to this conjecture, the functions $L_F^K$ and $L_F^{K,W}$ are equal for generic representations.

\begin{thme*}
 If the preceeding Conjecture is true, then the functions $L_F^K$ and $L_F^{K,W}$ are equal for all generic representations of the group $GL(n,K)$.
\end{thme*}

Now as Conjecture \ref{distprincgn} is proved in \cite{M1} for irreducible principal series representation. The preceeding theorem has as a consequence, the following result. 

\begin{thme**}
Let $\pi$ be an irreducible principal series representation of $GL(n,K)$. Then the functions $L_F^K(\pi)$ and $L_F^{K,W}(\pi)$ are equal.
\end{thme**}

Eventually, we prove the converse of Corollary \ref{asaigen}.

\begin{thme***}
 The equality of Asai $L$-functions of generic representations, and Asai $L$-functions of their Langlands parameter, is equivalent to the truth of Conjecture \ref{distgen}. 
\end{thme***}

The first part of Section 1 (Subsections 1.1 to 1.3) concerns itself with results of Bernstein and Zelevinsky about classification of generic representations of $GL(n,K)$ and derivatives of such representations (cf.\cite{BZ},\cite{Z}), results of Jacquet, Piatetski-Shapiro and Shalika about $L$-functions of pairs of Whittaker representations of $GL(n,K)$ (cf.\cite{JPS}), and basic results of Cogdell and Piatetski-Shapiro about Whittaker models of derivatives of Whittaker representations of $GL(n,K)$, and exceptional poles of $L$-functions of pairs.
The end of first section is a reminder of classical results about distinguished representations of $GL(n,K)$, and we state Conjecture \ref{distgen} about classification of distinguished generic representations.\\

In the second section, we first recall standard facts from \cite{F4} about Asai $L$-functions of generic representations such as their functional equation. Then we recall results from \cite{M2}, in particular Theorem \ref{pole2}, which asserts that a generic representation of $G_n(K)$ is distinguished if and only if its Rankin-Selberg type Asai $L$-function admits a pole at $0$.\\ 

The last section is devoted to adapt the method Cogdell and Piatetski-Shapiro (see \cite{CP}) for the computation, assuming Conjecture \ref{distgen}, of Asai $L$-functions of generic representations of $GL(n,K)$. The result is obtained in Theorem \ref{asailang} and its corollary, and has as an application, the computation of Asai $L$-functions of principal series representations of $GL(n,K)$.

\section{Preliminary results}

\subsection{Basic properties of studied groups and representations}

We fix until the end of this paper a local field $F$ of characteristic zero, and an algebraic closure $\bar{F}$ of $F$. All the finite extensions of $F$ we will consider will be subfields of $\bar{F}$. We fix a quadratic extension $K$ of $F$.\\

If $A$ is a ring, $M$ and $M'$ two $A$-modules, we denote by $Hom_A(M,M')$ the space of morphisms of $A$-modules from $M$ two $M'$. If $G$ is a group acting on two vector spaces $V$ and $V'$, then $Hom_G(V,V')$ designates the space of $G$-equivariant maps from $V$ to $V'$. \\
If $E$ is a field, 
 and if $E'$ is a Galois extension of $E$, we then denote by $Gal_E(E')$ 
the Galois group of $E'$ over $E$. We denote by $\sigma$ the non trivial element in $Gal_F(K)$.\\
If $E$ is a finite extension of $F$, we denote by $v_E$ the discrete valuation of $E$, which verifies that $v_E(\pi_E)$ is $1$ if $\pi_E$ is a prime element of $E$. We denote by $q_E$ the cardinality of the residual field of $E$. We denote by $|\ |_E$ the absolute value of $E$ defined by $|x|_E = q_E^{-v_E (x)}$, for $x$ in $E$. We denote by $R_E$ the valuation ring of $E$, and by $P_E$ the maximal ideal of $R_E$. Finally we denote by $W_E$ the Weil group of $E$ (cf. \cite{T}), and by $W'_E$ the Weil-Deligne group of $E$. The group $W'_E$ is the semidirect product group $W_E \rtimes SL(2,\mathbb{C})$, with $W_E$ acting by its quotient group $q_E^\mathbb{Z}$ on $SL(2,\mathbb{C})$, such that if we take a Frobenius element $\phi_E$ in $W_E$, the action of $\phi_E$ on $SL(2,\mathbb{C})$ is given by conjugation by the matrix $\left( \begin{array}{cc}
q_E & 0\\
0 & 1 \end{array}\right)$.

Let $G$ be an affine algebraic group defined on the field $F$, and if $E$ is an extension of $F$, we denote by $G(E)$ the group of the points of $G$ over $E$. Such a group is locally compact and totally disconnected, we will call it an $l$-group. More generally if $G$ is an algebraic group defined over $\mathbb{Z}$ and if $A$ is a ring, we denote by $G(A)$ the group of its points over $A$.\\ \\
Let $n$ be a positive integer, we denote by $M_n= M_n(\bar{F})$ the additive group of $n\times n$ matrices with entries in $\bar{F}$, and we denote by $G_n$ the general linear group $GL(n,\bar{F})$ of invertible matrices of $M_n(\bar{F})$. If $M$ belongs to $M_n$, we denote its determinant by $det(M)$.\\
We call partition af a positive integer $n$, a family $\bar{n}=(n_1,\dots,n_t)$ of positive integers (for a certain $t$ in $\mathbb{N}-\left\lbrace 0\right\rbrace $), such that the sum $n_1+\dots+n_t$ is equal to $n$. To such a partition, we associate an algebraic subgroup of $G_n$ denoted by $P_{\bar{n}}$, given by matrices of the form $$\left (\begin{array}{cccccccc}
g_1 & \star & \star & \star & \star \\
    & g_2 & \star &  \star & \star \\
    &  & \ddots & \star & \star \\  
    &  &     & g_{t-1} &  \star \\
    &  &     &   & g_t
 \end{array}\right),$$ with $g_i$ in $G_{n_i}$ for $i$ between $1$ and $t$. We call it the standard parabolic subgroup associated with the partition $\bar{n}$. We call parabolic subgroup any conjugate of a standard parabolic subgroup. We denote by $N_{\bar{n}}$ its unipotent radical subgroup, given by the matrices $$\left (\begin{array}{cccccccc}
I_{n_1} & \star & \star\\
     &  \ddots & \star\\  
    &  & I_{n_t}
 \end{array}\right)$$ and by $M_{\bar{n}}$ its Levi subgroup given by the matrices $$\left (\begin{array}{cccccccc}
g_{1} &   &  \\
     & \ddots  &  \\  
    &  & g_{t}
 \end{array}\right) $$ with the $g_i$'s in $G_{n_i}$. The group $P_{\bar{n}}$ identifies with the semidirect product $N_{\bar{n}}\rtimes M_{\bar{n}}$.\\
If the partition $\bar{n}$ is $(1,\dots,1)$, we denote by $B_n$ the group $P_{\bar{n}}$, which we call the standard Borel subgroup of $G_n$, we denote by $N_n$ its unipotent radical.\\

We denote by $P_n$ the subgroup of $G_n$ given by the matrices $\left (\begin{array}{cccccccc}
g & \star\\  
    & 1
 \end{array}\right)$, with $g$ in $G_{n-1}$, we call it the standard affine subgroup (sometimes called mirabolic subgroup) of $G_n$. We note $U_n$ its unipotent radical, given by matrices of the form $\left (\begin{array}{cccccccc}
I_{n-1} & V\\  
  0_{1,n-1}  & 1
 \end{array}\right)$.\\
We denote by $A_n$ the standard maximal torus of $G_n$ of diagonal matrices.\\

If $E$ is a finite extension of the field $F$, we denote by $K_n^E$ the standard maximal compact subgroup $G_n(R_E)$ of $G_n(E)$. This group admits a natural filtration by the compact open subgroups $K_{n,r}^E= I_n + M_n(P_E^r)$, for $r$ in $\mathbb{N}-\left\lbrace 0\right\rbrace$.\\



Let $G$ be an $l$-group, we denote by $d_G g$ or simply $dg$ if the context is clear, a left Haar measure $G$. For $x$ in $G$, we denote by $\Delta_G (x)$ the positive number defined by the relation $d_g (gx)= \Delta_G (x)d_g (g)$. The module $\Delta_G$ defines a morphism from $G$ into $\mathbb{R}_{>0}$. We denote by $\delta_G$ the morphism from $G$ into $\mathbb{R}_{>0}$ defined by $x \mapsto \Delta_G(x^{-1})$.\\
If $H$ is a closed subgroup of $G$, we denote by $d_{H\backslash G} g$ (resp.  $d_{G/H} g$) a right invariant (resp. left invariant) measure on the quotient space $H\backslash G$ (resp. $G/H$), we simply note it $dg$ if there is no ambiguity.\\
We usually note $d^*x$ the Haar measure when $x$ varies in the multiplicative group of a field, or in a product of such groups.\\
Let $\bar{n}$ be a partition of $n$ and $E$ a finite extension of $F$, the module of $P_{\bar{n}}(E)$ is given by $$\Delta_{P_{\bar{n}}}: \left( \begin{array}{cccccccc}
m_1 & \star & \star & \star & \star \\
    & m_2 & \star &  \star & \star \\
    &  & \ddots & \star & \star \\  
    &  &     & m_{t-1} &  \star \\
    &  &     &   & m_t
 \end{array}\right)  \mapsto \prod_{i=1}^t |det(m_i)|_E^{n_1+\dots+n_{i-1}-n_{i+1}-\dots-n_t}.$$

Let $G$ be an $l$-group, and $H$ a subgroup of $G$, a representation $(\pi,V)$ of $G$ is said to be smooth if for any vector $v$ of the vector space $V$, there is a subgroup $U_v$ of $G$ stabilizing $v$ through $\pi$. We denote by $V^H$ subspace of fixed points of $V$ under $H$. The category of smooth representations of $G$ is denoted by $Alg(G)$. If $(\pi,V)$ is a smooth representation of $G$, we denote by $\pi^{\vee}$ its dual representation in the smooth dual space $\tilde{V}$ of $V$.\\

\textbf{We will only consider smooth representations of $l$-groups.}\\

Let $X$ be a locally closed space of an $l$-group $G$, and $H$ closed subgroup of $G$, with $H.X \subset X$.
If $V$ is a complex vector space, we denote by $C^{\infty}(X,V)$ the space of smooth functions from $X$ to $V$, and by $C_c^{\infty}(X,V)$ the space of smooth functions with compact support from $X$ to $V$ (if one has $V=\mathbb{C}$, we simply denote it by $C_c^{\infty}(X)$).\\ 
If $\rho$ is a complex representation of $H$ in  $V_{\rho}$, we denote by $C^{\infty}(H \backslash X, \rho, V_{\rho})$ the space of functions $f$ from $X$ to $V_{\rho}$, fixed under the action by right translation of some compact open subgroup $U_f$ of $G$, and which verify $f(hx)=\rho(h)
 f(x)$ for $h \in H$, and $x \in X$ (if $\rho$ is a character, we denote this space by $C^{\infty}(H \backslash X, \rho)$. We denote by $C_c^{\infty}(H \backslash X, \rho, V_{\rho})$ subspace of functions with support compact modulo $H$ of $C^{\infty}(H \backslash X, \rho, V_{\rho})$.\\ 
We denote by $Ind _H^G (\rho)$ the representation by right translation of $G$ in $C^{\infty}(H \backslash G, \rho,
 V_{\rho})$ and by $ind _H^G (\rho)$ the representation by right translation of $G$ in $C_c^{\infty}(H \backslash G,  \rho,
 V_{\rho})$. We denote by ${Ind'} _H^G (\rho)$ the normalized induced representation $Ind _H^G ((\Delta _G /\Delta _H)^{1/2} \rho)$ and by ${ind'} _H^G (\rho)$ the normalized induced representation $ind _H^G ((\Delta _G /\Delta _H)^{1/2} \rho)$.\\ 
Let $n$ be a positive integer, and $\bar{n}=(n_1,\dots,n_t)$ be a partition of $n$, and suppose that we have a representation $(\rho_i, V_i)$ of $G_{n_i}(K)$ for each $i$ between $1$ and $t$. Let $\rho$ be the extension to $P_{\bar{n}}$ of the natural representation $\rho_1 \otimes \dots \otimes \rho_t$ of $M_{\bar{n}}\simeq G_{n_1}(K) \times \dots \times G_{n_t}(K)$, by taking it trivial on $N_{\bar{n}}$. We denote by $\rho_1 \times \dots \times \rho_t$ the representation ${Ind'} _{P_{\bar{n}}(K)}^{G_n (K)} (\rho)$.

\subsection{Classification of generic representations of $GL(n,K)$}

{From now on, we will assimilate a representation to its isomorphism class, hence two isomorphic representations can be said to be equal}.\\
We recall in this subsection classical results of Bernstein and Zelevinsky about classification of irreducible representations of the group $G_n(K)$.\\
If $\pi$ is an irreducible representation of $G_n(K)$, one denotes by $c_{\pi}$ its central character.\\
We recall that an irreducible representation of $G_n(K)$ is called supercuspidal if it doesn't occur as a quotient of any proper parbolically induced representation, which is equivalent to the fact that it has a coefficient with support compact modulo the center $Z_n(K)$ of the group $G_n(K)$.\\

An irreducible representation $\pi$ of the group $G_n(K)$ is called quasi-square-integrable, if there exists a positive character $\chi$ of the multiplicative group $K^*$, such that one of the coefficients $g\mapsto c(g)$ of $\pi$ verifies that $c(g)\chi(det(g))$ is a square-integrable function for a Haar measure of $G_n(K)/Z_n(K)$. One says that the representation $\pi$ is  square-integrable (or belongs to the discrete series of $G_n(K)$) if one can choose $\chi$ to be trivial.\\
If $\rho$ is a supercuspidal representation of $G_r(K)$ for a positive integer $r$, on denoteq by $\rho|\ |_K$ the representation obtained by twisting with the character $|det(\ )|_K$. We call segment a list $\Delta$ of supercuspidal representations of the form $$\Delta=[\rho,\rho|\ |_F,\dots,\rho|\ |_F^{l-1}] $$ for a positive integer $l$. We call length of the segment the integer $rl$. We denote by $\Delta$ the unique irreducible quotient of the representation $\rho \times \rho|\ |_F \times \dots \times \rho|\ |_F^{l-1}$ of $G_{rl}(K)$. It is also common to denote by $St_l(\rho)$ the representation $\Delta$, and to call it generalized Steinberg representation attached to $\rho$. 
We have the following theorem (Theorem 9.3 of \cite{Z}).

\begin{thm} A representation $\Delta$ of the group $G_n(K)$ is quasi-square-integrable if and only if there is $r\in \left\lbrace 1,\dots, n \right\rbrace$ and $l\in \left\lbrace 1,\dots, n \right\rbrace$ such that $lr=n$, and $\rho$ a supercuspidal representation of $G_r(K)$ such that the representation $\Delta$ is equal to $St_l(\rho)$.\end{thm}

 A representation of this type is  square-integrable if and only if it is unitarizable, or equivalently if and only if $\rho|\ |_F^{(l-1)/2}$ is unitarizable (i.e. its central character is unitary). We say that two segments are linked if none of them is a subsegment of the other, but their union is still a segment.\\

We recall that from Theorem 6.1 of \cite{Z}, that every irreducible representation of the group $G_n(K)$, is obtained as the unique irreducible quotient of a representation parabolically induced from quasi-square-integrables. We will focus in this work on representations called of Whittaker type, we thus recall their definition. Before, we recall that if $\psi$ is a character of $K$, we still denote by $\psi$ the character of the group $N_n(K)$, given by $n\mapsto \psi(\sum_{i=1}^{n-1} n_{i,i+1})$.\\

\begin{df}
We say that the representation $(\pi,V)$ of the group $G_n(K)$ is a representation of Whittaker type, if there is a non trivial character $\psi$ of $(K,+)$, such that the space of linear forms $\lambda$ on $V$, which verify $\lambda(\pi(n)v)=\psi(n)v$ for $n$ in $N_n(K)$ and $v$ in $V$, is of dimension $1$.
\end{df}

Hence a representation of Whittaker type is not necessarily irreducible, but it is isomorphic, up to unique (modulo scalars) isomorphism to a submodule of $Ind_{N_n(K)}^{G_n(K)}(\psi)$. We note $W(\pi,\psi)$ this model of $\pi$ on which $G_n(K)$ acts by right translation, and call it the Whittaker model of $\pi$.

The following theorem due to Rodier (cf. \cite{R}) describes the behaviour of the ``Whittaker type'' with respect to parabolic induction.

\begin{thm}\label{rod}
Let $\bar{n}=(n_1,\dots,n_t)$ be a partition of $n$, and let $\pi_i$ be a representation of $G_{n_i}(K)$ for $i$ between $1$ and $t$, then the representation $\pi_1 \times \dots \pi_t$ of the group $G_n(K)$ is of Whittaker type if and only if each $\pi_i$ is.

\end{thm}

We now study irreducible representations of Whittaker type.

\begin{df}
A representation $(\pi,V)$ of the group $G_n(K)$ is called generic if it is irreducible and of Whittaker type. 
\end{df}

We have the following theorem due to Zelevinsky (Th. 9.7 of \cite{Z}), which classifies the generic representations of the group $G_n(K)$:

\begin{thm}\label{classgen}
Let $\bar{n}=(n_1,\dots,n_t)$ be a partition of $n$, and let $\Delta_i$ be a quasi-square-integrable of $G_{n_i}(K)$ for $i$ between $1$ and $t$, the representation $\pi=\Delta_1\times \dots \times \Delta_t$ of the group $G_n(K)$ is irreducible if and only if no $\Delta_i$'s are linked, in which case $\pi$ is generic. If $(m_1,\dots,m_{t'})$ is another partition of $n$, and if the $\Delta'_j$'s are unlinked segments of length $m_j$ for $j$ between $1$ and $t'$, then the representation $\pi$ equals $ \Delta'_1 \times \dots \times \Delta'_{t'}$ if and only if $t=t'$, and $\Delta_i=\Delta'_{s(i)}$ for a permutation $s$ of $\left\lbrace 1,\dots,t\right\rbrace$. Eventually, every generic representation of $G_n(K)$ is obtained this way.
\end{thm}

In \cite{BZ}, Bernstein and Zelevinsky show the following result:

\begin{thm}\label{kirgn}
Let $\pi$ be a generic representation of $G_n(K)$, and $\psi$ a non trivial character of $K$, the restriction map to the affine subgroup $P_n(K)$ of functions of $W(\pi,\psi)$ is injective. 
\end{thm}

We denote by $K(\pi,\psi)$ the model of $\pi$ obtained this way, we call it its Kirillov model, it has the following important property.

\begin{thm}\label{kirgn2}
Let $\pi$ be a generic representation of the group $G_n(K)$, and $\psi$ a non trivial character of $K$, Then the kirillov model $K(\pi,\psi)$ contains the space $C_c^{\infty}(N_n(K) \backslash P_n(K), \psi)$.

\end{thm}

\subsection{Whittaker models and derivatives}

We start by recalling results of \cite{BZ} and \cite{Z}. We define four functors $\Phi^{+}$, $\Phi^{-}$, $\Psi^{+}$, $\Psi^{-}$:\\
$\bullet$ $\Psi^{+}$ is a functor from $Alg(G_{n-1}(K))$ into $Alg(P_n(K))$. If $(\sigma, V)$ belongs to $Alg(G_{n-1}(K))$, the space of $\Psi^{+}(\sigma)$ is $V$, and $U_n(K)$ acts trivially by $\Psi^{+}(\sigma)$ whereas $G_{n-1}(K)$ acts by $\Psi^{+}(\sigma)(g)=|det(g)|^{1/2} \sigma(g)$.\\
$\bullet$ $\Psi^{-}$ is a functor from $Alg(P_n(K))$ into $Alg(G_{n-1}(K))$. If $(\tau, W)$ is in $Alg(P_n(K))$, the space of $\Psi^{-}(\tau)$ is the quotient of $W$ by the $G_{n-1}(K)$-module $$W(U_n,1)=\left\lbrace \tau(u)w-w / w \in W, u\in U_n(K) \right\rbrace,$$ and $$\Psi^{-}(\tau)(g)(w+W(U_n,1))= |det(g)|^{-1/2}\tau(g)w+W(U_n,1)$$ for $g$ in $G_{n-1}(K)$.\\
$\bullet$ $\Phi^{+}$ is a functor from $Alg(P_{n-1}(K))$ into $Alg(P_n(K))$. If $(\sigma, V)$ belongs to $Alg(P_{n-1}(K))$, we extend $\sigma$ to $P_{n-1}(K)U_n(K)$ by making $U_n(K)$ act by the character $\psi$. We denote by $\sigma \otimes \psi$ this extension, then we have $$\Phi^{+}(\sigma)= ind_{P_{n-1}(K) U_n(K)}^{P_n(K)}(|\det(\ )|^{1/2}\sigma \otimes \psi).$$
$\bullet$ $\Phi^{-}$ is a functor from $Alg(P_n(K))$ into $Alg(P_{n-1}(K))$. If $(\tau, W)$ is in $Alg(P_n(K))$, the space of $\Phi^{-}(\tau)$  is the quotient of $W$ by the $P_{n-1}(K)$-module $$W(U_n,\psi)=\left\lbrace \tau(u)w-\psi(u)w / w \in W, u\in U_n(K) \right\rbrace,$$ and one has $$\Phi^{-}(\tau)(p)(w+W(U_n,\psi))= |det(p)|^{-1/2}\tau(p)w+W(U_n,\psi)$$ for $p$ in $P_{n-1}(K)$.\\

These four functors are exact and verify the following relations:
\begin{enumerate}\label{functorrel}
 \item $\Phi^{-}\Phi^{+}\simeq Id$ and $\Psi^{-}\Psi^{+}\simeq Id$

\item $\Phi^{-}\Psi^{+}=0$ and $\Psi^{-}\Phi^{+}=0$

\item $0\rightarrow \Phi^{+}\Phi^{-}\rightarrow  Id \rightarrow \Psi^{+}\Psi^{-}\longrightarrow 0$ is exact.

\end{enumerate}

The functors $\Psi^{+}$ and $\Phi^{+}$ send irreducible representations to irreducible representations. Any irreducible representation of $P_n(K)$ is of the form $(\phi^{+})^{k-1}\psi^{+}(\rho)$ for an irreducible representation $\rho$ of $G_{n-k}(K)$. 
For $\tau$ in $Alg(P_n(K))$, we denote by $\tau_{(k)}$ the representation $(\Phi^{-})^k(\tau)\in Alg(P_{n-k}(K))$ and $\tau^{(k)}$ the representation $\Psi^{-}(\Phi^{-})^{k-1}(\tau)=\Psi^{-}(\tau_{(k-1)}) \in Alg(G_{n-k}(K))$.
The representation $\tau^{(k)}$ is called the $k^{th}$ derivative of $\tau$.\\ 
If $\pi$ belongs to $Alg(G_{n}(K))$, we denote by $\pi^{(k)}$ the $k^{th}$ derivative of $\pi$ viewed as a $P_n(K)$-module, and we call it the $k^{th}$ derivative of $\pi$ again.\\
A representation $\tau$ in $Alg(P_n(K))$ admits a filtration of $P_n(K)$-modules $$0 \subset \tau_n \subset \tau_{n-1} \subset \dots \subset \tau_2 \subset \tau_1=\tau$$ with $\tau_k=(\Phi^{+})^{k-1}(\Phi^{-})^{k-1}(\tau)$. Hence the quotient $\tau_k/\tau_{k+1}$ is isomorphic to $$(\Phi^{+})^{k-1}\Psi^{+}(\tau^{(k)}).$$
Let $\psi$ be a non trivial character of $K$, the representation $\tau$ admits a nonzero Whittaker linear form with respect to $\psi$ if and only if $\tau_n$ is nonzero, and the dimension of the space of these linear forms is the dimension of $\tau^{(n)}$. In particular, the representation $\tau$ admits a Whittaker model if and only if $\tau^{(n)}$ is of dimension $1$, in which case all the representations $\tau_k$ admit a Whittaker model.\\ 
For an irreducible representation $\pi$ of the group $G_n(K)$, the dimension of $\pi^{(n)}$ is at most $1$, so that the representation $\pi$ is generic if and only if $\pi^{(n)}$ is nonzero.\\

The following proposition allows to compute the derivatives of generic representations (\cite{Z}, Theorem 4.4 and Lemma 4.5).

\begin{prop}\label{calculder}

\begin{enumerate}
 \item Let $\pi$ be a supercuspidal representation of the group $G_n(K)$, then the representation $\pi^{(k)}$ is zero for $k$ between $1$ and $n-1$, and $\pi^{(n)}=1$.

\item Let $\rho$ be a supercuspidal representation of $G_r(K)$, and $\pi$ a quasi-square-integrable representation $[\rho,\rho|\ |_F,\dots,\rho|\ |_F^{l-1}] $ of $G_{rl}(K)$. Then $\pi^{(k)}$ is zero if $k$ is not a multiple of $r$, and $\pi^{(kr)}=[\rho|\ |_F^k,\dots,\rho|\ |_F^{l-1}] $ for $k$ between $1$ and $l-1$.

\item Let $\bar{n}=(n_1,\dots,n_t)$ be a partition of $n$, and for $i$ between $1$ and $t$, let $\Delta_i$ be an irreducible representation of $G_{n_i}(K)$. We denote by $\pi$ the representation $\Delta_1 \times \dots \times \Delta_t$ of the group $G_n(K)$, then for $k$ between $1$ and $n$, the representation $\pi^{(k)}$ has a composition series, with factors the $\Delta_1^{(k_1)}\times \dots \times \Delta_t^{(k_t)}$ for $k_1+\dots+k_t= n$.

\end{enumerate}
 
\end{prop}

We end this subsection by recalling few results on Whittaker models and derivatives.

\begin{prop}{(\cite{CP}, proposition 1.2)}
Let $\pi$ be a representation of Whittaker type of the group $G_n(K)$, the for all $k$ between $1$ and $n$, the representation $\pi_{(k-1)}$ of $P_{n-k+1}(K)$, admits for model $$ W(\pi_{(k-1)},\psi)=\left\lbrace g\mapsto |det(g)|_K^{(1-k)/2} W \left( \begin{array}{cc} g &  \\
  & I_k
\end{array} \right)  | W \in W(\pi,\psi), g \in G_{n-k}(K)
 \right\rbrace $$ which we call the Whittaker model of $\pi_{(k-1)}$. 

\end{prop}

In this model, Proposition 1.4. of \cite{CP} asserts that the subspace $$ W(\pi_{(k-1)},\psi)(U_{n-k+1},1)=\left\lbrace \pi(u)W - W | u \in U_{n-k+1} , W \in W(\pi_{(k-1)},\psi) \right\rbrace $$ is equal to $$ \aligned \lbrace g \mapsto |det(g)|_K^{(1-k)/2} W \left( \begin{array}{cc} g &  \\
  & I_k
\end{array} \right) \text{| there exists } N>0  \text{ such that si } \underset{i}{max} |g_{n-k,i}|<q^{-N},\\ \text{ then }  W \left( \begin{array}{cc} g &  \\
  & I_k
\end{array} \right) =0 \rbrace . \endaligned $$ 

Now let $\pi_{0}^{(k)}$ be an irreducible subrepresentation of the representation $\pi^{(k)}$ with model the quotient normalised representation $W(\pi_{(k-1)},\psi)/W(\pi_{(k-1)},\psi)(U_{n-k+1},1)$. We denote by $\tau_0$ the inverse image of $\pi_{0}^{(k)}$ in $\pi_{(k-1)}$, and $W(\tau_0,\psi)$ the corresponding subspace $W(\pi_{(k-1)},\psi)$ (which contains $W(\pi_{(k-1)},\psi)(U_{n-k+1},1)$).\\

The following proposition implies that $\pi_{0}^{(k)}$ is generic.

\begin{prop}{(\cite{CP}, Proposition 1.7)}\label{whitder}
Let $\pi$ be a generic representation of the group $G_n(K)$, let $k$ be an integer $1$ and $n-1$, and $\pi_{0}^{(k)}$ an irreducible subspace of $\pi^{(k)}$ with central character $\omega_{0}$. One defines $\tau_{0}$ as above, then for any function $F$ in $W(\tau_0,\psi)$, the function $$a\mapsto \omega_{0}^{-1}(a)|a|_K^{(k-n)/2} F\left( \begin{array}{cc} a I_{n-k} &  \\
  & I_k
\end{array} \right)$$ is constant when $a\in K$ is in a sufficiently small neighbourhood of zero. The linear form on $W(\tau_0,\psi)$ defined by $$F \mapsto \lim\limits_{a \to 0} \omega_{0}^{-1}(a)|a|_K^{(k-n)/2} F\left( \begin{array}{cc} a I_{n-k} &  \\
  & I_k
\end{array} \right),$$ factors through the quotient by $W(\pi_{(k-1)},\psi)(U_{n-k+1},1)$, and defines a (nonzero) Whittaker linear form on the Whittaker model of $\pi_{0}^{(k)}$. 
\end{prop}

A consequence of this proposition is the following

\begin{cor}\label{whittder}{(\cite{CP}, corollary of Proposition 1.7)}
Let $\pi$ be a generic representation of the group $G_n(K)$, let $k$ be an integer between $1$ and $n-1$, and $\pi_{(0)}^{(n-k)}$ a sub-$G_{k}(K)$-module of $\pi^{(n-k)}$. We denote by $\tau_0$ the inverse image of $\pi_{(0)}^{(n-k)}$ in $\pi_{(n-k-1)}$ for the natural projection from $\pi_{(n-k-1)}$ on $\pi^{(n-k)}$. Then if $\psi$ is a non trivial character of $K$, for all $W_0$ in $W(\pi_{(0)}^{(n-k)},\psi)$, there exists $W$ in $W(\tau_0,\psi)$ and $\Phi_0 \in C_c^{\infty}(K^k)$, nonvanishing at zero such that
$$ W_0(g) \Phi_0(e_kg)= |det(g)|^{-(n-k)/2} W \left( \begin{array}{cccccccc}
h &  \\
  & I_{n-k}
 \end{array}\right)\Phi_0(e_kg).$$ 
Moreover, for all $W$ in $W(\tau_0,\psi)$ and for all $\Phi_0$ in $C_c^{\infty}(K^k)$, with support sufficiently small around zero, there exists $W_0$ in $W(\pi_{(0)}^{(n-k)},\psi)$ such that
$$  W \left( \begin{array}{cccccccc}
h &  \\
  & I_{n-k}
 \end{array}\right)\Phi_0(e_kg)= |det(g)|^{(n-k)/2} W_0(g) \Phi_0(e_kg).$$
\end{cor}

\subsection{Functional equations for representations of Whittaker type}

We first recall that an Euler factor is a function of the complex variable $s$, of the form $1/P(q_K^{-s})$ (or $1/P(q_F^{-s})$ later in the framework of Asai $L$ functions) with $P$ in $\mathbb{C}[X]$ and $P(0)=1$.\\

We now summerize the results of Section 2 of \cite{JPS}.\\
Let $\pi$ and $\pi'$ be two representations of Whittaker type of the group $G_n(K)$,
  and $\psi$ a non trivial character of $K$. Let  $W$ belong $W(\pi,\psi)$, $W'$ belong to $W(\pi',\psi^{-1})$,
 and $\phi$ be in $C_c^{\infty}(K^n)$. We denote by $\eta_n$ the row vector with $n$ entries $(0,\dots,0,1)$. Then, there exists a real number $r_{\pi,\pi'}$ independant of $(W,W',\phi)$, such that the following integral is absolutely convergent whenever $Re(s)\geq r_{\pi,\pi'}$: $$I(W,W',\phi,s)=\int_{G_n(K)/N_n(K)}W(g)W'(g)\phi(\eta_n g)|det(g)|_K^s dg.$$

The integrals $I(W,W',\phi,s)$ 
 define for $Re(s)$ large enough rational functions in $q_K^s$ and $q_K^{-s}$, which have a Laurent series developpement in $q_K^{-s}$. Their analytic extensions define meromorphic functions on $\mathbb{C}$. The $\mathbb{C}$-vector space generated by the functions $I(W,W',\phi,s)$
 for $(W,W',\phi)$ in $W(\pi,\psi)\times W(\pi',\psi^{-1})\times C_c^{\infty}(K^n)$
  is actually a fractional ideal of $\mathbb{C}[q_K^s,q_K^{-s}]$, having a unique generator which is an Euler facteur. We denote by $L(\pi \times \pi',s)$
 this generator which is independant of $\psi$. \\

We denote by $w_n$ the element $\left( \begin{array}{ccc}
 &  & 1 \\
  & \adots & \\ 
 1 &  &  
 \end{array}\right)$ of $G_n(F)$.\\
	If $\pi$ is a representation of Whittaker type of the group $G_n(K)$, and if $W$ belongs to $W(\pi,\psi)$, then the function $\tilde{W}$ sending $g$ in $ G_n (K)$ to $W(w_n {}^t\! g^{-1})$ belongs to $W(\pi^{\vee},\psi^{-1})$. If $\psi$ is a non trivial character of a finite extension $E$ over $F$, it defines naturally a character of $E^n$ that we denote by $\psi$ again, we denote by $d_{\psi} x$ the autodual Haar measure on $E^n$. If $\phi$ belongs to $C_c^{\infty}(E^n)$, we denote by $\widehat{\phi}^{\psi}$ or by $\widehat{\phi}$ if there is no ambiguity, the Fourier transform $ y\mapsto \int_{x\in E^n}\phi(x)\psi(xy)d_{\psi} x$.  

\begin{prop}[Functional equations for representations of Whittaker type]\label{eqgn}

Let $\psi$ be a non trivial character of $(K,+)$, $\pi$ and $\pi'$ two representations of Whittaker type of the group $G_n(K)$.
 If $(W,W',\phi)$ belongs to $W(\pi,\psi)\times W(\pi',\psi^{-1})\times C_c^{\infty}(K^n)$, there exists a factor $\epsilon(\pi \times \pi', \psi, s)$, 
 complex multiple of a (maybe negative) power of $q_K^s$, such that the following functional equation is satisfied:
$$\frac{ I(\tilde{W}, \tilde{W'},\widehat{\phi}, 1 -s)}{ L(\pi^{\vee}\times {\pi'}^{\vee},1-s)}  = c_{\pi'}(-1)^{n-1}\epsilon(\pi \times \pi', \psi, s) \frac{ I(W, W',\phi, s)}{ L(\pi \times {\pi'} ,s)}.$$ 

We then put $$\gamma(\pi \times \pi',\psi,s)= c_{\pi'}(-1)^{n-1}\epsilon(\pi \times \pi', \psi, s)\frac{L(\pi^{\vee}\times {\pi'}^{\vee},1-s)}{L(\pi \times {\pi'} ,s)}.$$ 
 \end{prop}

We now recall the notion of exceptional pole introduced in \cite{CP}.\\
Let  $\pi$ and $\pi'$ be two representations of Whittaker type of the group $G_n(K)$, and let $s_0$ be a pole of order $d$ of $L(\pi \times \pi',s)$, the Laurent developement $$ I(W, W',\phi, s)=T_{s_0}(W,W',\phi)/(q_K^s - q_K^{s_0})^d+ \ higher \ order \ terms $$
defines a linear form on $W(\pi,\psi)\times W(\pi',\psi^{-1})\times C_c^{\infty}(K^n)$ satisfying the quasi-invariance: $$T_{s_0}(\pi(g)W,\pi'(g)W',\rho(g)\phi,s)=|det(g)|_K^{-s_0}T_{s_0}( W, W', \phi,s).$$

\begin{df}\label{polexpaire}
In the preceeding situation, we say that a pole $s_0$ of $L(\pi \times \pi',s)$ is exceptional if the trilinear form $T_{s_0}$ vanishes on $W(\pi,\psi)\times W(\pi',\psi^{-1})\times C_c^{\infty}(K^n -\left\lbrace 0\right\rbrace) $.
\end{df}

It is shown in \cite{CP} that if $\pi$ and $\pi'$ are irreducible, and if $s_0$ is an exceptional pole of $L(\pi \times \pi',s)$, then $\pi^{\vee}=|\ |_K^{s_0} \pi'$. A proof analoguous to that of Theorem 2.1 of \cite{M2}, appealing to Theorem A of \cite{Ber} instead  of Proposition 1.1 of \cite{M2}, would show the converse of this fact.

\begin{thm}\label{distsplit}
Let  $\pi$ and $\pi'$ be two generic representations of the group $G_n(K)$, then the function $L(\pi \times \pi',s)$ admits an exceptional at $s_0$ if and only if we have $\pi^{\vee}=|\ |_K^{s_0} \pi'$. 
\end{thm}

In the same context, we call $L_{ex}(\pi\times \pi',s)$ the product of the $1/(1-q^{(s_0-s)})^d$ when $s_0$ describes the set of the exceptional poles of $L(\pi \times \pi',s)$, where $d$ is the ordrer of $s_0$ in $L(\pi \times \pi',s)$. We denote by $L_{Rad(ex)}(\pi\times \pi',s)$ the product of the $1/(1-q^{(s_0-s)})$ when $s_0$ describes the set of the exceptional poles of $L(\pi \times \pi',s)$.

\subsection{Classical results and conjectures about distinguished representations}\label{sectconj}

We first define the notion of distinction.

\begin{df}\label{dfdist}
Let $\mu$ be a character of $F^*$, a representation $(\pi,V)$ of the group $G_n(K)$ is said to be $\mu$-distinguished if there exists on $V$ a nonzero linear form $L$, satisfying the relation $L(\pi(g)v)=\mu(det(g))L(v)$ for $g$ in $G_n(F)$ and $v$ in $V$. We say that $\pi$ is distinguished if the character $\mu$ is trivial.
\end{df}

If $\pi$ is a homomorphism from a group $G$ to another group, and $\phi$ a homomorphism from $G$ to itself, we denote by $\pi^{\phi}$ the homomorphism $\pi \circ \phi$. The irreducible distinguished representations have two fondamental properties.

\begin{prop}{(\cite{F2}, proposition 11)}
Let $\pi$ be an irreducible distinguished representation of the group $G_n(K)$, then the dimension of the space of $G_n(F)$-invariant linear forms on the space of $\pi$ is $1$. 
\end{prop}

\begin{prop}{(\cite{F2}, proposition 12)}\label{autodual}
Let $\pi$ be an irreducible distinguished representation of the group $G_n(K)$, then $\pi^{\vee}$ is isomorphic to $\pi^{\sigma}$.
\end{prop}

We denote by $\eta_{K/F}$ the order two $2$ character of $F^*$, trivial on the norms of the multiplicative group  $K^*$. Kable showed in the discrete series case, some sort of converse to the preceeding theorem.

\begin{prop}\label{distdiscr}
Let $\pi$ be a discrete series representation of the group $G_n(K)$ such that the dual representation $\pi^{\vee}$ is isomorphic to $\pi^{\sigma}$, then $\pi$ is distinguished or $\eta_{K/F}$-distinguished. A representation of the discrete series of the group $G_n(K)$ cannot be distinguished and $\eta_{K/F}$-distinguished at the same time.
\end{prop}

We find a local proof in \cite{AT} (Proposition 3.1) for the case of the group $G_2(K)$, the proof in the general case uses global methods, and is due to Kable (cf. \cite{K}).\\

A consequence of Proposition \ref{autodual} and Theorem \ref{classgen} is the following. If $\pi=\Delta_1 \times \dots \times \Delta_t$ is a generic representation as in the statement of Theorem \ref{classgen} and if it is distinguished, then there exists a reordering of the ${\Delta _i}$'s, and an integer $r$ between $1$ and $t/2$, such that $\Delta_{i+1}^{\sigma} = \Delta_i^{\vee} $ for $i=1,3,..,2r-1$, and $\Delta_{i}^{\sigma} = \Delta_i ^{\vee}$ for $ i > 2r$. According to Proposition \ref{distdiscr}, this means that there exists a reordering of the ${\Delta _i}$'s, and an integer $r$ between $1$ and $t/2$, such that $\Delta_{i+1}^{\sigma} = \Delta_i^{\vee} $ for $i=1,3,..,2r-1$, and such that $\Delta_{i}$ is distinguished or $\eta_{K/F}$-distinguished for $i > 2r$. The following statement is actually conjectured (cf. \cite{F3}):

\begin{conj}\label{distgen} 
Let $\pi=\Delta_1 \times \dots \times \Delta_t$ a generic representation of the group $G_n(K)$ as in Theorem \ref{classgen}, it is distinguished if and only if il there is a reordering of the ${\Delta _i}$'s, and an integer $r$ between $1$ and $t/2$, such that $\Delta_{i+1}^{\sigma} = \Delta_i^{\vee} $ for $i=1,3,..,2r-1$, and $\Delta_{i}$ is distinguished for $i > 2r$.
\end{conj}

\section{Conjectures and known results on Rankin-Selberg type Asai L-functions}\label{sectasaibasic}

We first recall standard results from \cite{F4}.\\
We fix a nonzero element $\delta$ in the kernel of $Tr_{K/F}$. A non trivial character $\psi$ of $(K,+)$, trivial on $F$ corresponds uniquely to a character $\psi'$ of $F$, by $\psi(x)=\psi'(Tr_{K/F}(\delta x))$ for $x$ in $K$.\\
Let $\pi$ be a representation of Whittaker type of the group $G_n(K)$. If $W$ belongs to the Whittaker model $W(\pi,\psi)$, for $\psi$ a nontrivial character of $K$ trivial on $F$, and  $\phi$ belongs to $C_c^{\infty}(F^n)$, then there exists a real number $r_{\pi}$, depending only on $\pi$, such that the following integral converges:

$$I(W,\phi,s)=\int_{N_n(F)\backslash G_n(F)} W(g)\phi(\eta_n g){|det(g)|_F}^s dg.$$

This integral as a function of $s$ has a meromorphic extension to
 $\mathbb{C}$.\\
For $s$ of real part greater than $r_{\pi}$, the function $I(W,\phi,s)$ is a
 rational function in $q_F^{-s}$, which actually has a Laurent series
 development.\\
The $\mathbb{C}$-vector space generated by these functions is in fact
 a fractional ideal $I(\pi)$ of $\mathbb{C}[q_F^{-s},q_F^s]$.
This ideal $I(\pi)$ is principal, and has a unique generator of the form
 $1/P(q_F^{-s})$, where $P$ is a polynomial with $P(0)=1$.

\begin{df}
We denote by $L_F^K(\pi,s)$ the generator of $I(\pi)$ defined just above,
 and call it the Asai $L$-function of $\pi$.
\end{df}


 For $\phi$ in $C_c^{\infty}(F^n)$, we denote by $\widehat{\phi}$ the function $\widehat{\phi}^{\psi'}$, we then have the following functional equation.\\

\begin{thm}{(\textbf{Functional equation})}{(Th. of \cite{F4})}\label{Asaieq}

Let $\pi$ be a representation of Whittaker type of the group $G_n(K)$ and $\psi$ a non trivial character of $(K,+)$, trivial on $F$, corresponding to the character $\psi'$ by the fixed $\delta$. There exists an epsilon factor $\epsilon_F^K(\pi,\psi' ,s)$ which
 is, up to scalar, a (maybe negative) power of $q_F^s$, such that the
 following functional equation is satisfied for any $W$ in
 $W(\pi,\psi)$ and
 any $\phi$ in $C_c^{\infty}(F^n)$:

$$I(\tilde{W},\widehat{\phi}^{\psi'},1-s)/L_F^K(\pi^{\vee},1-s) =
 c_{\pi}(-1)^{n-1} \epsilon_F^K(\pi,\psi',s)I(W,\phi,s)/L_F^K(\pi,s).$$

\end{thm}

We denote by $\gamma_F^K(\pi,\psi' ,s)$ the function $c_{\pi}(-1)^{n-1} \epsilon_F^K(\pi,\psi',s)L_F^K(\pi^{\vee},1-s)/L_F^K(\pi,s).$\\

Now suppose $L_F^K(\pi,s)$ has a pole at $s_0$, its order $d$ is the
 highest order pole of the family of functions of $I(\pi)$.\\
Then we have the following Laurent expansion at $s_0$:  

\begin{equation}\label{dvlaurent} I(W,\phi,s)=
 B_{s_0}(W,\phi)/(q_F^s-q_F^{s_0})^d + \ higher \ order \
 terms\end{equation}

The quantity $B_{s_0}(W,\phi)$ defines a nonzero bilinear form on
 $W(\pi,\psi) \times C_c^{\infty}(F^n)$, satisfying the quasi-invariance:

$$ B_{s_0}(\pi(g)W,\rho(g)\phi)= |det(g)|_F^{- s_0} B_{s_0}(W,\phi), \ g\in G_n(F).$$

Following \cite{CP} for the split case $K=F\times F$, we state the following definition:

\begin{df} A pole $s_0$ of the Asai $L$-function $L_F^K(\pi,s)$ is called exceptional if the
 associated bilinear form $B_{s_0}$ vanishes on $W(\pi,\psi) \times
  C_c^{\infty}(F^n - \left\lbrace 0 \right\rbrace)$. \end{df}

As an immediate consequence, if $s_0$ is an exceptional pole of
 $L_F^K(\pi,s)$, then $B_{s_0}$ is of the form $B_{s_0}(W,\phi)=
 \lambda
 _{s_0}(W) \phi(0)$, where $\lambda _{s_0}$ is a nonzero $|det(\ )|_F^{-
 s_0}$-invariant linear form on $W(\pi,\psi)$.\\




For more convenience, we introduce a second $L$-function: for $W$ in $W(\pi,\psi)$, by standard arguments, the following integral
 is convergent for $Re(s)$ large, and defines a rational function in $q^{-s}$, which has a Laurent series development in $q^{-s}$:

$$I_{(0)}(W,s)=\int_{N_n(F)\backslash P_n(F)} W(p){|det(p)|_F}^{s-1} dp.$$

By standard arguments again, the vector space generated by the
functions $I_{(0)}(W,s)$, for $W$ in $W(\pi,\psi)$, is a fractional ideal $I_{(0)} (\pi)$ of
 $\mathbb{C}[q_F^{-s},q_F^s]$, which has a unique generator which is an Euler factor.

\begin{df}\label{defL0}
Let $\pi$ be a representation of the group $G_n(K)$ which is of Whittaker type, we denote by $L_{(0)}(\pi,s)$ the Euler factor which generates the ideal spanned by the functions $I_{(0)}(W,s)$ for $W$ in $W(\pi,\psi)$. 

\end{df}

We now recall Lemma 2.1 of \cite{M2}.

\begin{LM}\label{LM}(Compare to \cite{JPS} p. 393)

Let $W$ be in $W(\pi,\psi)$, one can choose $\phi$ with support small enough around $(0,\dots,0,1)$ such that $I(W,\phi,s)=I_{(0)}(W,s)$.
 
\end{LM}



Hence we have the inclusion $I_{(0)} (\pi) \subset I(\pi)$, which implies
 that $L_{(0)}(\pi,s)= L_F^K(\pi,s) R(q_F^s,q_F^{-s})$ for some $R$ in $\mathbb{C}[q_F^{-s},q_F^s]$.
But because $L_{(0)}$ and $L_F^K$ are both Euler factors, $R$ is actually
 just a polynomial in $q_F^{-s}$, with constant term equal to one.
Denoting by $L_{Rad(ex)} (\pi,s)$ its inverse (which is an Euler factor), we have
 $L_F^K(\pi,s)=L_{(0)}(\pi,s)L_{Rad(ex)} (\pi,s)$, we will say that $L_{(0)}$ divides
 $L_F^K$. 

We now recall Proposition 2.2 of \cite{M2} which gives a characterisation of exceptional poles:

\begin{prop}
 
A pole of $L_F^K(\pi,s)$ is exceptional if and only if it is a pole of the function $L_{Rad(ex)} (\pi,s)$ defined just above.

\end{prop}

We denote by $L_{ex}(\pi,s)$ (as in \cite{CP} for the split case, cf. Definition \ref{polexpaire}) the product function $\prod_{s_i} (1-q_F^{s_i-s})^{d_i}$, where the $s_i$'s are the exceptional poles of $L_F^K(\pi,s)$ and the $d_i$'s their order in $L_F^K(\pi,s)$.


 



Theorem 2.1 of \cite{M2} then links the notions of exceptional pole and distinction.

\begin{thm}\label{pole2} 
A generic representation $\pi$ of the group $G_n(K)$ is
 distinguished if and only if the function $L_F^K(s,\pi)$ admits an exceptional
 pole at zero.

\end{thm}

For discrete series representations, the previous theorem has the simpler form:

\begin{prop}\label{K}(\cite{K}, Theorem 4)\\
 
A discrete series representation $\pi$ of the group $G_n(K)$ is
 distinguished if and only if its Asai $L$-function $L_F^K(s,\pi)$ admits a pole at zero.

\end{prop}

Let $s_0$ be in $\mathbb{C}$.
We notice that if the representation $\pi$ is a generic representation of the group $G_n(K)$, it is $|\ |_{F}^{-s_0}$-distinguished if and only if the representation $\pi \otimes |\ |_{K}^{s_0 /2}$ is distinguished, but as $L_F^K(\pi \otimes |\ |_{K}^{s_0 /2},s)$ is equal to $L_F^K(\pi,s +s_0)$, Theorem \ref{pole2} becomes:

\begin{thm}\label{s-dist}

A generic representation $\pi$ of the group $G_n(K)$ is
 $|\ |_{F}^{-s_0}$-distinguished if and only if the function $L_F^K(\pi, s)$ admits an exceptional
 pole at $s_0$.
 
\end{thm}




 It is shown in Remark 2.1 of \cite{M2} that if $\pi$ is a generic representation of $G_n(K)$, then the function $L_{Rad(ex)}(\pi,s)$ has simple poles, the following proposition specifies this statement.

\begin{prop}\label{LRad(ex)}{(cf. \cite{M2}, Proposition 2.5)}

Let $\pi$ be an generic representation of the group $G_n(K)$, then the Euler factor $L_{Rad(ex)}(\pi,s)$ is equal to $\prod 1/(1-q_F^{s_0-s})$ where the product is taken over the exceptional poles $s_0$ of the function $L_F^K(\pi,s)$, i.e. the $q_F^{s_0}$'s such that the representation $\pi$ is $|\ |_F^{-s_0}$-distinguished. In particular, the function $L_{Rad(ex)}(\pi,s)$ divides $L_{ex}(\pi,s)$.

\end{prop}

Suppose now that the representation $\pi$ is supercuspidal, then the restriction to $P_n(K)$ of any $W$ in $W(\pi,\psi)$ has compact support modulo $N_n(K)$, hence $I_{(0)}(W,s)$ is a polynomial in $q^{-s}$, and $L_{(0)}(\pi,s)$ is equal to $1$. Hence Proposition \ref{LRad(ex)} implies: 

\begin{prop}\label{Lcusp}
Let $\pi$ be an irreducible supercuspidal representation of the group $G_n(K)$, it is $|\ |_F^{-s_0}$-distinguished if and only if its Asai L-function $L_F^K(\pi,s)$ has a pole at $s_0$.
\end{prop}

We end this paragraph by stating a conjecture relating the Asai $L$-function $L_F^K$ of generic a representation and a certain canonical $L$ function also denoted by $L_F^K$ of the corresponding representation (through Langlands correspondence) of the Weil-Deligne group of $K$.\\
If $\rho$ is a finite dimensional representation of the Weil-Deligne group $W'_K$, we denote by
 $M_{W'_K}^{W'_F}(\rho)$ the representation of $W'_F$ induced
 multiplicatively from $\rho$. We recall its definition:\\
If $V$ is the space of $\rho$, then the space of
 $M_{W'_K}^{W'_F}(\rho)$ is $V\otimes V$. Denoting by $\tau$ an element of $W_F -W_K$, and $\sigma$ the element
 $(\tau, I_2)$ of $W'_F$, we have:\\

$$ M_{W'_K}^{W'_F}(\rho)(h)(v_1 \otimes v_2) = \rho(h) v_1 \otimes
 \rho^{\sigma}(h) v_2$$ for $h$ in $W'_K$, $v_1$ and $v_2$ in $V$.

$$ M_{W'_K}^{W'_F}(\rho)(\sigma)(v_1 \otimes v_2) = \rho(\sigma ^2) v_2
 \otimes v_1$$ for $v_1$ and $v_2$ in $V$.\\

We refer to section 7 of \cite{P} for the definition and the basic properties
 of multiplicative induction in general.

\begin{df}
The function $L_F^K(\rho,s)$ is by definition the usual $L$-function of the representation $M_{W'_K}^{W'_F}(\rho)$, i.e. $L_F^K(\rho,s)=L(M_{W'_K}^{W'_F}(\rho),s)$. If $\pi$ is the irreducible representation of $G_n(K)$, associated to $\rho$ by the local Langlands correspondence, we denote by $L_F^{K,W}(\pi,s)$ the function $L_F^K(\rho,s)$.
\end{df}

The following conjecture is expected to be true.

\begin{conj}\label{conjasai}
Let $\pi$ be a generic representation of the group $G_n(K)$, and let $\rho$ be the representation of dimension $n$ of the Weil-Deligne group $W'_K$ of $K$, corresponding to $\pi$ through Langlands correspondence. Then we have the following equality of $L$-functions:
$$ L_F^K(\pi,s) = L_F^K(\rho,s) .$$
\end{conj}

Indeed, this conjecture is true for discrete series representations, it follows from Theorem 1.6 of \cite{AR} and Theorem of Section 1.5 in \cite{He}. The proof uses global methods. For $G_2(K)$, the conjecture is proved in \cite{M2} for ordinary representations, using local methods.

\section{Asai $L$-functions of generic representations}

\textbf{The aim of this section is to show that Conjecture \ref{distgen} of the first section implies the conjecture \ref{conjasai} about Asai L-functions of generic representations.}\\

Assuming Conjecture \ref{distgen}, we will first show in Subsection \ref{sectradex} that Conjecture \ref{conjasai} is true for representations in general position (see Definition \ref{gen}). Then in Subsection \ref{subsectionder}, we will deform a representation $\pi$ in general position by twisting it by an unramified character of the center of a standard Levi subgroup, identified with a complex multivariable $u\in(\mathbb{C}/\frac{2i \pi}{ln(q_K)\mathbb{Z}})^t$, into a representation $\pi_u$. We will then prove that the equality of Conjecture \ref{conjasai} is preserved when $u$ varies. This will allow us to conclude.\\

For irreducible principal series representations (i.e. if the ${\Delta _i}$'s are characters), the conjecture \ref{distgen} is proved in \cite{M1}, hence this will allow us to compute Asai $L$-functions of such representations.\\

\subsection{Asai $L$-functions of quasi-square-integrable representations}\label{sectasaidiscr}

We start this subsection by computing Asai $L$-functions of quasi-square-integrable representations.

\begin{prop}\label{asaidiscrete}
Let $\rho$ be an irreducible supercuspidal representation of $G_m(K)$, and $St_n(\rho)$ the representation $\pi=[\rho|\ |_K^{(1-n)/2},\rho|\ |_K^{(3-n)/2},\dots,\rho|\ |_K^{(n-3)/2},\rho|\ |_K^{(n-1)/2}]$ of $G_{mn}(K)$. We have the following equality:
\begin{equation}\label{egalasaidiscrete} L_F^K(St_n(\rho),s)=\prod_{k=0}^{n-1} L_F^K(\eta_{K/F}^{n-1-k}\rho,s+k).\end{equation}
\end{prop}

\begin{proof} We know from Proposition Theorem 1.6 of \cite{AR} and Theorem of Paragraph 1.5 in \cite{He} that $L_F^K(\pi,s)$ is equal to the standard $L$-function of $M_{W'_K}^{W'_F}(\rho \otimes Sp(n))$. This latter representation is isomorphic to $M_{W'_K}^{W'_F}(\rho)\otimes M_{W'_K}^{W'_F}(Sp(n))$. Here $SL(2,\mathbb{C})$ acts trivially on $M_{W'_K}^{W'_F}(\rho)$, and $W_F$ acts as $M_{W_K}^{W_F}(\rho)$. The $SL(2,\mathbb{C})$-module $M_{W'_K}^{W'_F}(Sp(n))$ is isomorphic to $Sp(n)\otimes Sp(n)$, which is the direct sum $\oplus_{k=1}^{n} Sp(2k-1)$. In this direct sum, the submodules $Sym(Sp(n))$ and $Alt(Sp(n))$ are respectively isomorphic to $\oplus_{k=0}^{E((n-1)/2)} Sp(2n-1-4k)$ and $\oplus_{k=0}^{E(n/2-1)} Sp(2n-3-4k)$.\\
 As $W_F$ acts trivially on $Sym(Sp(n))$, and by $\eta_{K/F}$ on $Alt(Sp(n))$, we deduce the equality $$M_{W'_K}^{W'_F}(\rho \otimes Sp(n))=\oplus_{k=0}^{n-1} \rho \otimes \eta_{K/F}^{k} Sp(2n-1-2k).$$ 
The wanted equality of Asai $L$-functions follows. \end{proof}

This has for corollary the following result (see Section 4 of \cite{AR}).

\begin{cor}
Let $\rho$ be an irreducible supercuspidal representation of $G_m(K)$. The representation $\pi=St_n(\rho)$ of $G_{mn}(K)$ is distinguished if and only if $\rho$ is $\eta_{K/F}^{n-1}$-distinguished.
\end{cor}

\begin{proof} Suppose $\rho$ is $\eta_{K/F}^{n-1}$-distinguished, then it is unitary because its central character has unitary restriction to $F^*$, hence is unitary itself. Thus $\pi$ is a discrete series representation, and its Asai $L$-function has a pole at zero from Proposition \ref{asaidiscrete}, the representation $\pi$ is then distinguished according to Proposition \ref{K}.\\
On the other hand, if the representation $\pi$ is distinguished, its central character $c_{\pi}$ is trivial on $F^*$, hence unitary, but as $c_{\pi}=c_{\rho}^n$, the central character $c_{\rho}$ is also unitary, and the supercuspidal representation $\rho$ is unitary too. Again $\pi$ is a discrete series representation, and its Asai $L$-function has a pole at zero because it is distinguished according to Proposition \ref{K}. From equality \ref{egalasaidiscrete}, there is $k$ between $0$ and $n-1$ such that $L_F^K(\eta_{K/F}^{n-1-k}\rho,s+k)$ has a pole at zero, which implies from Proposition \ref{Lcusp} that $\rho$ is $\eta_{K/F}^{n-1-k}|\ |_F^{-k}$-distinguished, but as $c_{\rho}$ is unitary, the integer $k$ must be equal to zero, and $\rho$ is $\eta_{K/F}^{n-1}$-distinguished. \end{proof}

\begin{cor}\label{polediscrete}
If $\Delta$ is a discrete series representation of the group $G_n(K)$, the function $L_F^K(\Delta,s)$ has no pole in the half plane $\left\lbrace s/ Re(s)>0 \right\rbrace$. 
\end{cor}
\begin{proof}
There is a unitary supercuspidal representation $\rho$ of $G_m(K)$, and a divisor $d$ of $n$ such that $\Delta = St_d(\rho)$. According to Proposition \ref{asaidiscrete} and Proposition \ref{Lcusp}, the function $L_F^K(\Delta,s)$ has a pole at $s_0$ if and only if there is an integer $k$ between $0$ and $d-1$ such that $\rho$ is $\eta_{K/F}^{n-1-k}|\ |_F^{-s_0 -k}$ distinguished. The central character of $\rho$ being unitary, the real part of $s_0$ must be equal to $-k$.\end{proof}

\begin{cor}\label{divdiscrete}
If $\Delta$ is a quasi-square-integrable representation of the group $G_n(K)$, and $k$ is an integer between $1$ and $n-1$, such that $\Delta^{(k)}$ is nonzero, then the function $L_F^K(\Delta^{(k)},s)$ divides $L_F^K(\Delta,s)$.
\end{cor}
\begin{proof} It is a consequence of Proposition \ref{calculder} and Proposition \ref{asaidiscrete}.\end{proof}

\begin{rem} In the same manner, it follows from Proposition \ref{calculder} and Theorem 8.2 of \cite{JPS} (or Theorem 2.3 of \cite{CP}) that if $\Delta$ and $\Delta'$ are respectively two quasi-square-integrable representations of $G_n(K)$ and $G_n'(K)$, then if $k$ and $k'$ are two integers such that $\Delta^{(k)}$ and $\Delta'^{(k')}$ are nonzero, the function $L(\Delta^{(k)}\times \Delta'^{(k')},s)$ divides the function $L(\Delta \times \Delta' ,s)$.\end{rem}

\subsection{Asai $L$-functions and derivatives}

In this subsection, we analyse the function $L_{(0)}(\pi,s)$ for $\pi$ in ``general position''. This is a straightforward transposition of the results of \cite{CP}, to the context of Asai L-functions.\\

 To do this we write the Laurent expansion of $I_{(0)} (W,s)$, for $W$ in $\mathcal{W}(\pi,\psi)$. Suppose that $L_{(0)}(\pi,s)$ (see definition \ref{defL0}) has a pole of order $d$ at $s_0$:

$$I_{(0)}  (W,s)= L_{s_0}(W,s)/(q_F^s-q_F^{s_0})^d + \ higher \ order \ terms.$$

Here the linear form $L_{s_0}$ is a non trivial linear form on $W(\pi,\psi)$, which satisfies the quasi-invariance $L_{s_0}(\pi(p)W)= |det(p)|_F^{1- s_0} L_{s_0}(W)$ for $p$ in $P_n(F)$ and $W$ in $W(\pi,\psi)$.\\

Considering the filtration of $P_n(K)$-submodules $W(\pi_n,\psi) \subset ... \subset W(\pi_2,\psi) \subset W(\pi_1,\psi)=W(\pi,\psi)$, then $L_{s_0}$ must be zero on $W(\pi_n,\psi)$, because restrictions to $P_n(K)$ of functions $W$ in $W(\pi_n,\psi)$ have compact support modulo $N_n(K)$, hence functions $I_{(0)}  (W,s)$ belong to $\mathbb{C}[q_F^{-s}]$. Therefore there exists a smallest $k$ such that $L_{s_0}$ vanishes on $W(\pi_{n-k+1},\psi)$ but is non-zero on $W(\pi_{n-k},\psi)$.\\
Hence the functions $I_{(0)}(W,s)$ with $W$ in $W(\pi_{n-k},\psi)$ account for the pole at $s_0$ with order $d$. From Section 1 of \cite{CP}, functions $W$ of $W(\pi_{n-k},\psi)$ viewed as functions $W(\begin{array}{lll} g &  \\  & 1 \end{array})$ on $G_{n-1}(K)$ have compact support in the last $n-k-1$ rows of $g$ modulo $N_{n-1}(K)$. Using partial Iwasava decomposition to write $g$ in $G_{n-1}(K)$ as $$g=\left
 (\begin{array}{cccccccc}
h &  &  & \\
    & a_{k+1} &  &     \\
    &  & \ddots &   \\  
    &  &     & a_{n-1} 
 \end{array}\right)k' \ (mod \ N_{n-1}(K))$$ with $h$ in $G_k(K)$, $k'$ in $K_{n-1}^K$, the function $W$ has compact multiplicative support in the $a_i$. Then there exists constants $c_i$ and $\beta_i$ such that our integral becomes a finite sum of the form $$ I_{(0)}  (W,s)= \sum_i c_i q^{-\beta_i s} \int_{N_k(F)\backslash G_k(F)} W_i \left(\begin{array}{lll} h &  \\  & I_{n-k}\end{array}\right)  |det(h)|_F^{s-(n-k)} dh.$$  
Set \begin{equation}\label{eqder}I_{(n-k-1)}(W,s)= \int_{N_k(F)\backslash G_k(F)} W\left(\begin{array}{lll} h &  \\  & I_{n-k}\end{array}\right)  |det(h)|_F^{s-(n-k)} dh.\end{equation} 
The pole at $s_0$ of order $d$ must come from an integral $I_{(n-k-1)}(W,s)$. Morover since $W$ enters into these integrals via its restriction to $G_k(F)$, it only depends on its image in $W(\pi_{(n-k-1)},\psi)$. But by Lemma 9.2 of \cite{JPS}, each $I_{(n-k-1)}(W,s)$ actually occurs in $I_{(0)} (W_0,s)$ for some $W_0$ in $W(\pi,\psi)$, hence belongs to $I_{(0)} (\pi)$. These integrals have therefore a pole of order at most $d$ at $s_0$, and a pole of order $d$ for appropriate choice of $W$ in $W(\pi,\psi)$.\\
We now write the expansion of $I_{(n-k-1)}(W,s)$ at $s_0$ for $W$ in $W(\pi,\psi)$: $$ I_{(n-k-1)}(W,s)=L_{(n-k-1),s_0}(W)/(q_F^s-q_F^{s_0})^d + \ higher \ order \ terms.$$
The value at $W\in W(\pi,\psi)$ of the linear form $L_{(n-k-1),s_0}$  depends only of the image of $W$ in $W(\pi_{(n-k-1)},\psi)$. Because of the twists involved in the definitions of the derivatives, the linear form $L_{(n-k-1),s_0}$ still verifies the quasi-invariance under the subgroup $P_{k+1}(F)$: $$L_{(n-k-1),s_0}(\pi_{(n-k-1)}(p)(W))=|det(p)|_F^{1- s_0} L_{(n-k-1),s_0}(W).$$ Furthermore the index $k$ was chosen so that $L_{(n-k-1),s_0}$ is non trivial but vanishes on $W(\pi_{n-k+1},\psi)$, i.e. on $W(\pi_{(n-k-1),2},\psi)$ if we consider it as a linear form on $W(\pi_{(n-k-1)},\psi)$.\\
 As the $P_{k+1}(K)$-module $\pi_{(n-k-1)}/\pi_{(n-k-1),2}$ is isomorphic to $\Psi^{+}(\pi^{(n-k)})$, the functionnal $L_{(n-k-1),s_0}$ defines a non trivial linear form on $W(\Psi^{+}(\pi^{(n-k)}),\psi)$.\\
We now make the same assumption as in \cite{CP}, which is verified for representations in ``general position''. 

\begin{ass}\label{ass} For all $j$ between $1$ and $n$, the derivative $\pi^{(n-j)}$ is completely reducible.\end{ass}

Taking $k$ as before, such that $L_{(n-k-1),s_0}$ defines a non trivial linear form on $W(\Psi^{+}(\pi^{(n-k)}),\psi)$. Then the representation $\pi^{(n-k)}= \oplus \pi_i^{(n-k)}$, with $\pi_i^{(n-k)}$ irreducible, and there exists $i_0$ such that $L_{(n-k-1),s_0}$ restrict non trivially to $W(\Psi^{+}(\pi_{i_0}^{(n-k)}),\psi)$. Now from Corollary \ref{whittder} of Proposition \ref{whitder}, if $W_{i_0}$ belongs to $W(\Psi^{+}(\pi_{i_0}^{(n-k)}),\psi)$, there exists a function $W'_{i_0}$ and a function $\phi_0$ in $C_c^{\infty}(K)$ which is the characteristic function of a sufficiently small neighborhood of $0$, such that for $h$ in $G_k(K)$,
\begin{equation} \label{res} W_{i_0} \left(\begin{array}{lll} h &  \\  & I_{n-k}\end{array}\right) \phi_0(\eta_k h)= W'_{i_0} (h)|det(h)|_K^{(n-k)/2}\phi_0(\eta_k h).\end{equation}
The integral $I_{(n-k-1)}(W_{i_0},s)$ decomposes as $I_{(n-k-1)}(W_{i_0},s)=I_{(n-k-1)}^0(W_{i_0},s)+I_{(n-k-1)}^1(W_{i_0},s)$, where $$I_{(n-k-1)}^0(W_{i_0},s)= \int_{N_k(F)\backslash G_k(F)} W_{i_0}\left(\begin{array}{lll} h &  \\  & I_{n-k}\end{array}\right)\phi_0(\eta_k h) |det(h)|_F^{s} dh$$ and $$I_{(n-k-1)}^1(W_{i_0},s)= \int_{N_k(F)\backslash G_k(F)} W_{i_0}\left(\begin{array}{lll} h &  \\  & I_{n-k}\end{array}\right)(1-\phi_0(\eta_k h)) |det(h)|_F^{s} dh.$$
As the $(1-\phi_0(\eta_k h))$ term restricts the support of the integrand to being compact in the last row of $h$, modulo $N_k(F)$, the integral $I_{(n-k-1)}^1(W_{i_0},s)$ is actually a linear combination of integrals involving restrictions of functions of $W(\pi,\psi)$ to $G_{k-1}(F)$, which depend only on their image in $W(\pi_{(n-k)},\psi)$. But the linear form $L_{(n-k-1),s_0}$ vanishes on this space, hence the term $I_{(n-k-1)}^1(W_{i_0},s)$ cannot contribute to a pole of order $d$ at $s_0$.\\ 
We deduce that integrals $I_{(n-k-1)}^0(W_{i_0},s)= \int_{N_k(F)\backslash G_k(F)} W'_{i_0}(h) \phi_0(\eta_k h)|det(h)|_F^{s} dh$ contribute to a pole of order $d$ at zero. These are integrals defining the Asai L-function of $\pi_{i_0}^{(n-k)}$, and as $\phi_0(0)=1$, the pole $s_0$ of order $d$ is an exceptional pole of $L_F^K(\pi_{i_0}^{(n-k)},s)$. Moreover, from equation \ref{res}, any integral for $\pi_{i_0}^{(n-k)}$ corresponding to an exceptional pole comes from a $I_{(n-k-1)}^0(W_{i_0},s)$ for some good choice of $W_{i_0}$.\\
We thus proved:

\begin{prop} Suppose that all the derivatives of the generic representation $\pi$ are completely reducible. Then any pole of $L_{(0)}(\pi,s)$ occurs as an exceptional pole of some $L_F^K(\pi_{i_0}^{(n-k)},s)$ with same order, for some integer $k$ between $1$ and $n-1$, with $\pi_{i_0}^{(n-k)}$ an irreducible constituent of $\pi^{(n-k)}$. Furthermore any exceptional pole of $L_F^K(\pi_{i_0}^{(n-k)},s)$, for $k$ an integer between $1$ and $n-1$ and $\pi_{i_0}^{(n-k)}$ an irreducible constituent of $\pi^{(n-k)}$, occurs with at least same order in $L_{(0)}(\pi,s)$.\end{prop}

Put in another way, we proved the following:

\begin{thm}\label{L0} Let $\pi$ be a generic representation, such that its derivatives are completely reducible, then the Euler factor $L_{(0)}(\pi,s)$ is equal to the least common multiple $\vee_{i,k} L_{ex}(\pi_i^{(n-k)},s)$, where the l.c.m. is taken over $k$ in $\left\lbrace 1,\dots,n-1 \right\rbrace$ and $\pi_i^{(n-k)}$ in the irreducible components of $\pi^{(n-k)}$. This implies that $L_F^K(\pi,s)$ is equal to the least common multiple $\vee_{i,k} L_{ex}(\pi_i^{(n-k)},s)$, where the l.c.m. is taken over $k$ in $\left\lbrace 0,\dots,n \right\rbrace$ and $\pi_i^{(n-k)}$ in the irreducible components of $\pi^{(n-k)}$ \end{thm}

\begin{cor}\label{minimalex} a) Let $\pi$ be a generic representation of $G_n(K)$, such that its derivatives are completely reducible, then if the complex number $s_0$ is a pole of $L_F^K(\pi,s)$, there is a minimal $k$, such that $s_0$ is a pole of $L_{ex}(\pi_i^{(k)},s)$, for an irreducible component $\pi_i^{(k)}$ of $\pi^{(k)}$.\\
b) If $\pi'$ is a generic representation of $G_{n'}(K)$, for a pole $s_0$ of $L(\pi\times \pi',s)$, there is a minimal $k$, such that $s_0$ is a pole of $L_{ex}(\pi_i^{(k)}\times \pi_j^{(k)},s)$, for some irreducible components $\pi_i^{(k)}$ and $\pi_j^{(k)}$ of $\pi^{(k)}$. \end{cor}

\begin{proof} For a), the assertion follows from Theorem \ref{L0}. Assertion b) follows from Theorems 1.1 and 2.2 of \cite{CP}.\end{proof}

\begin{cor}\label{L0discr} If $\Delta=St_m(\rho)$ is a generalized Steinberg representation of $G_{mr}(K)$ associated to the supercuspidal representation $\rho$ of $G_{r}(K)$, then one has $L_{(0)}(\Delta)=L_F^K(\Delta^{(r)})$. In the same manner, if $\Delta'=St_m'(\rho')$ is a generalized Steinberg representation of $G_{m'r}(K)$ associated to the supercuspidal representation $\rho'$ of $G_{r}(K)$, then one has $L_{(0)}(\Delta\times \Delta')=L_F^K(\Delta^{(r)}\times \Delta'^{(r)})$.\end{cor}
\begin{proof} It is a consequence of Theorem \ref{L0} and Proposition \ref{calculder}.\end{proof}

\subsection{$\mathbf{L_{Rad(ex)}}$ for generic representations}\label{sectradex}

In this subsection, we prove that Conjecture \ref{distgen} implies Conjecture \ref{conjasai} about Asai L-functions for representations in general position (see Definition \ref{gen}).\\

We first translate Conjecture \ref{distgen} in terms of functions $L_{Rad(ex)}$ of the ${\Delta _i}$'s. We will need some notations.\\
For each integer $r\leq t/2$, we denote by $A(r)$ the set of subsets of $\left\lbrace 1,\dots,t\right\rbrace$ with $2r$ elements. If $I_r$ belongs to $A(r)$, we denote by $P(I_r)$ the set of partitions of $I_r$ in pairs $\left\lbrace i_1,j_1 \right\rbrace \amalg \dots \amalg \left\lbrace i_r,j_r \right\rbrace $, and we denote by $I_r^c$ the complement of $I_r$ in $\left\lbrace 1,\dots,t\right\rbrace$. Finally if $P_r=\left\lbrace i_1,j_1 \right\rbrace \amalg \dots \amalg \left\lbrace i_r,j_r \right\rbrace$ belongs to $P(I_r)$, we denote by $\left\lbrace P_r\right\rbrace$ the set $\left\lbrace \left\lbrace i_1,j_1 \right\rbrace, \dots, \left\lbrace i_r,j_r \right\rbrace \right\rbrace$. 

\begin{prop}\label{distgen2} 
The following conjecture is equivalent to the conjecture \ref{distgen}. Conjecture: Let $\pi=\Delta_1 \times \dots \times \Delta_t$ be a generic representation of the group $G_n(K)$ as in Theorem \ref{classgen}, then the Euler factor $L_{Rad(ex)}(\pi,s)$ is equal to the lcm for 
$r$ between $1$ and $t/2$, $I_r$ in $A(r)$ and $P_r$ in $P(I_r)$ of the functions $\wedge_{\left\lbrace i,j\right\rbrace \in \left\lbrace P_r \right\rbrace} L_{Rad(ex)}(\Delta_{i}\times \Delta_{j}^{\sigma},s) \wedge_{i' \in I_r^c} L_{Rad(ex)}(\Delta_{i'},s)$. This gives the formula
\begin{equation}L_{Rad(ex)}(\pi,s)=\label{radex}\vee_{r\leq t/2}\vee_{I_r \in A(r)} \vee_{P_r \in P(I_r)}(\wedge_{\left\lbrace i,j\right\rbrace \in \left\lbrace P_r \right\rbrace} L_{Rad(ex)}(\Delta_{i}\times \Delta_{j}^{\sigma},s) \wedge_{i' \in I_r^c} L_{Rad(ex)}(\Delta_{i'},s)).\end{equation}
\end{prop}
\begin{proof} Assume the preceeding conjecture. According to Theorem \ref{pole2}, the representation $\pi$ is distinguished if and only if $L_{Rad(ex)}(\pi,s)$ admits a pole at zero. From equation \ref{radex}, it is equivalent to say that there is an integer $r\leq t/2$, a set $I_r$ with $2r$ elements, and a partition with $\left\lbrace i_1,j_1 \right \rbrace \amalg \dots \amalg \left\lbrace i_r,j_r \right\rbrace$ of $I_r$ such that for every $k$ between $1$ and $r$, the function $L_{Rad(ex)}(\Delta_{i_k}\times \Delta_{j_k}^{\sigma},s)$ has a pole at zero, and for every $i'$ in $\left\lbrace 1,\dots,t\right\rbrace -\left\lbrace i_1,j_1 , \dots, i_r,j_r \right\rbrace$ the function $L_{Rad(ex)}(\Delta_{i'},s)$ has a pole at zero. This is nothing else than the statement of Conjecture \ref{distgen} according to Theorem \ref{distsplit} and Proposition \ref{K}.\\
Now suppose that the conjecture \ref{distgen} is true. Then tensoring by the character $|\ |_K^{s_0/2}$, for a complex number $s_0$, the more general (though equivalent) statement is also true: the representation $\pi=\Delta_1 \times \dots \times \Delta_t$ is $|\ |_F^{-s_0}$ distinguished if and only if there exists a reordering of the ${\Delta _i}$'s, and $r \leq t/2$, such that $\Delta_{i+1}^{\sigma} = |\ |_K^{-s_0}\Delta_i^{\vee} $ for $i=1,3,..,2r-1$, and $\Delta_{i}$ is $|\ |_F^{-s_0}$-distinguished for $ i > 2r$. Hence $s_0$ is a pole of $L_{Rad(ex)}(\pi,s)$ if and only if there exists an integer $r \leq t/2$, a set $I_r \in A_r$, and a partition in pairs $P_r \in P(I_r)$, such that $s_0$ is a pole of $L_{Rad(ex)}(\Delta_{i}\times \Delta_{j}^{\sigma},s)$ for every $\left\lbrace i,j\right\rbrace$ in $\left\lbrace P_r \right\rbrace$, and a pole of $L_{Rad(ex)}(\Delta_{i'},s)$ for every $i'$ in $I_r^c$. Hence both functions $L_{Rad(ex)}(\pi,s)$ and $$\vee_{r\leq t/2}\vee_{I_r \in A(r)} \vee_{P_r \in P(I_r)}(\wedge_{\left\lbrace i,j\right\rbrace \in \left\lbrace P_r \right\rbrace} L_{Rad(ex)}(\Delta_{i}\times \Delta_{j}^{\sigma},s) \wedge_{i' \in I_r^c} L_{Rad(ex)}(\Delta_{i'},s))$$ have the same poles, but as both have simple poles and are Euler factors, they are equal.\end{proof}

\begin{rem} The Asai $L$-function of a quasi-square-integrable representation $\Delta$ of the group $G_n(K)$ is computed in Proposition \ref{asaidiscrete} of Subsection \ref{sectasaidiscr}, it has simple poles, and so is the case for $L(\Delta \times \Delta',s)$ for quasi-square-integrable representations $\Delta$ and $\Delta'$ (see for example Theorem 2.3 of \cite{CP}). Hence in the statement of Conjecture \ref{distgen2}, the functions $L_{Rad(ex)}(\Delta_{i}\times \Delta_{j}^{\sigma},s)$ and $L_{Rad(ex)}(\Delta_{i'},s)$ are respectively equal to $L_{ex}(\Delta_{i}\times \Delta_{j}^{\sigma},s)$ and $L_{ex}(\Delta_{i'},s)$.\end{rem}



We next define what we mean by ``in general position'' for a representation of the group $G_n(K)$. There is actually a slight difference between our definition and the definition in \cite{CP} (we add condition 3.).

\begin{df}\label{gen}
Let $\pi=\Delta_1 \times \dots \times \Delta_t$ be a representation of the group $G_n(K)$, obtained by normalized parabolic induction of quasi-square-integrable representations $\Delta_i$ of $G_{n_i}(K)$, for positive integers $n_i$ such that their sum is equal to $n$.\\ We say that it is in general position, if it verifies the following properties: 

\begin{enumerate}
\item It is generic (hence irreducible).
\item For all $k$ between $1$ and $n-1$, the central characters of the irreducible components of $\pi^{(n-k)}$ are different. Thus, for such a representation, the assumption \ref{ass} is satisfied.
\item If $(i,j,k)$ are three different integers between $1$ and $t$, and $(l_i,l_j,l_k)$ three integers in $\{0,\dots,n_i-1\} \times \{0,\dots,n_j-1\}\times \{0,\dots,n_k-1\}$, then the functions $L(\Delta_i^{(l_i)}\times (\Delta_j^{(l_j)})^{\sigma})$ and $L(\Delta_i^{(l_i)}\times (\Delta_k^{(l_k)})^{\sigma})$ have no common poles, and the functions \\ $L(\Delta_i^{(l_i)}\times (\Delta_j^{(l_j)})^{\sigma})$ and $L_F^K(\Delta_i^{(l_i)})$ have no common poles.
 
\end{enumerate}
\end{df}

We are now able to state the following theorem.

\begin{thm}\label{asgen} Assume Conjecture \ref{distgen}. Let $\pi=\Delta_1 \times \dots \times \Delta_t$ be a generic representation of the group $G_n(K)$ as in Theorem \ref{classgen}, and suppose that it is in general position, then the following equality holds:
\begin{equation}\label{lasgen}L_F^K(\pi,s)= \prod_{1\leq i<j \leq t} L(\Delta_i \times \Delta_{j}^{\sigma},s)\prod_{1\leq k\leq t}L_F^K(\Delta_k,s)\end{equation}
 \end{thm}
\begin{proof} We first recall from Proposition \ref{calculder} that if the representation $\pi_i^{(n-k)}$ is an irreducible component of the derivative $\pi^{(n-k)}$ of $\pi=(\Delta_1 \times \dots \times \Delta_n)$ in general position, it is of the form $\Delta_{i_1}^{(t_1)} \times \dots \times \Delta_{i_k}^{(t_n)}$ for a subset $\left\lbrace i_1,\dots, i_k\right\rbrace$ of $\left\lbrace 1,\dots, n\right\rbrace$, hence it is still in general position.
We proove the theorem \ref{asgen} by induction on $n$. Case $n=1$ is trivial, and $n=2$ is treated in \cite{M2}.\\
Suppose that the theorem is true for any positive integer strictly less than $n$. By definition of $L_{(0)}$ and $L_{Rad(ex)}$, we have $L_F^K(\pi,s) =L_{(0)}(\pi,s)L_{Rad(ex)}(\pi,s)$. From Theorem \ref{L0} and Proposition \ref{distgen2}, this is equivalent to the function $L_F^K(\pi,s)$ being equal to $$[\vee_{l,k} L_{ex}(\pi_l^{(n-k)},s)][\vee_{r\leq t/2}\vee_{I_r \in A(r)} \vee_{P_r \in P(I_r)}(\wedge_{\left\lbrace i,j\right\rbrace \in \left\lbrace P_r \right\rbrace} L_{ex}(\Delta_{i} \times \Delta_{j}^{\sigma},s) \wedge_{i' \in I_r^c} L_{ex}(\Delta_{i'},s))] $$ for $k$ in $\left\lbrace 1,\dots,n-1 \right\rbrace$ and $\pi_l^{(n-k)}$ in the irreducible components of $\pi^{(n-k)}$.\\
By induction hypothesis, any pole of $L_{(0)}(\pi,s)$ is a pole of $$\prod_{1\leq i<j \leq t} L(\Delta_i \times \Delta_j^{\sigma},s)\prod_{1\leq k\leq t}L_F^K(\Delta_k,s).$$ The same holds for poles of $L_{Rad(ex)}(\pi,s)$ according to Equation \ref{lasgen}.\\
Hence to prove Equation \ref{lasgen}, as we are dealing with Euler factors, it suffices to prove that any pole of order $d$ of $\prod_{1\leq i<j \leq t} L(\Delta_i \times \Delta_{j}^{\sigma},s)\prod_{1\leq k\leq t}L_F^K(\Delta_k,s)$ occurs with same order in $L_F^K(\pi,s)$.\\
Let $s_0$ be such a pole, because $\pi$ is in general position, and because the $L_F^K(\Delta_k,s)$'s and the $L(\Delta_i \times \Delta_j^{\sigma},s)$'s have simple poles, there are two integers $p$ and $q$ such that the pole $s_0$ occurs exactly in $p$ functions $$L(\Delta_{i_1} \times \Delta_{j_1}^{\sigma},s),\dots, L(\Delta_{i_p} \times \Delta_{j_p}^{\sigma},s),$$ and $q$ functions $$L_F^K(\Delta_{k_1},s),\dots, L_F^K(\Delta_{k_q},s),$$ where $p+q=d$ and $I_{2p+q}=\left\lbrace i_1,j_1,\dots,i_p,j_p,k_1,\dots,k_q \right\rbrace$ is a subset of cardinality $2p+q$ of $\left\lbrace 1,\dots,t \right\rbrace$.\\
From Corollary \ref{minimalex}, there exist minimal integers $a_i$'s, $b_j$'s and $c_k$'s, such that $s_0$ is a pole of each function $$L_{ex}(\Delta_{i_1}^{(a_1)} \times (\Delta_{j_1}^{(b_1)})^{\sigma},s),\dots, L_{ex}(\Delta_{i_p}^{(a_p)} \times (\Delta_{j_p}^{(b_p)})^{\sigma},s),$$ and each function $$L_{ex}(\Delta_{k_1}^{(c_1)},s),\dots, L_{ex}(\Delta_{k_q}^{(c_q)},s).$$

There are two cases:\\
1) If $2p+q < t$ or if an element of the family $(a_i,b_j,c_k)_{i,j,k}$ is positive:\\
Then by induction hypothesis, the product of the $L_{ex}(\Delta_{i_r}^{(a_r)} \times (\Delta_{j_r}^{(b_r)})^{\sigma},s)$'s and of the $L_{ex}(\Delta_{k_{r'}}^{(c_{r'})},s)$'s for $r$ and $r'$ respectiveley between $1$ and $p$ and $1$ and $q$, divides the function $$L_F^K(\Delta_{i_1}^{(a_1)} \times \Delta_{j_1}^{(b_1)}\times \dots\times \Delta_{i_p}^{(a_p)} \times \Delta_{j_p}^{(b_p)}\times \Delta_{k_1}^{(c_1)}\times \dots \times \Delta_{k_q}^{(c_q)}).$$
Hence $s_0$ is a pole of order $d$ of this function, and it is exceptional from Proposition \ref{distgen2}.\\
Proposition \ref{L0} then implies that $s_0$ is a pole of order at least $d$ of $L_{(0)}(\pi,s)$.\\
 Suppose $\pi_i^{(l)}$ is an irreducible component of $\pi^{(l)}$ with $l$ positive, then it is of the form \\
$\Delta_1^{(d_1)}\times \dots \times \Delta_t^{(d_t)},$ with the sum of the $d_i$'s being equal to $l$. Now by induction hypothesis, the function $L_{ex}(\Delta_1^{(d_1)}\times \dots \times \Delta_t^{(d_t)},s)$ divides the function $\prod_{1\leq i<j \leq t} L(\Delta_i^{(d_i)} \times (\Delta_{j}^{(d_j)})^{\sigma},s)\prod_{1\leq k\leq t}L_F^K(\Delta_k^{(d_k)},s)$, which in turn divides $\prod_{1\leq i<j \leq t} L(\Delta_i \times \Delta_{j}^{\sigma},s)\prod_{1\leq k\leq t}L_F^K(\Delta_k,s)$ because of Corollary \ref{divdiscrete}. Hence the order of $s_0$ in $L_{ex}(\pi^{(l)})$ is at most $d$.\\
 We just proved, according to Proposition \ref{L0}, that the order of $s_0$ in $L_{(0)}(\pi)$ is at most $d$, hence $d$.\\
 As the pole $s_0$ doesn't occur in $L_{Rad(ex)}(\pi)$ because of Proposition \ref{distgen2}, it has order $d$ in $L_F^K(\pi)$.\\

2) The integer $2p+q $ is equal to $t$ and all the $a_i$'s, $b_j$'s and $c_j$'s are zero:\\
We first reorder the $\Delta_i$'s such that $i_r= 2r-1$, $j_r= 2r$ for $r\leq p$, and $k_{r'}= 2p+r'$ for $r'\leq q$. \\
In this case $s_0$ is an axceptional pole of $L_F^K(\pi,s)$ from Proposition \ref{distgen2}. It thus occurs with order $1$ in the function $L_{Rad(ex)}(\pi,s)$.\\
 Considering the $d-1$ last functions of $$L_{ex}(\Delta_{1} \times \Delta_{2}^{\sigma},s),\dots, L_{ex}(\Delta_{2p-1} \times \Delta_{2p}^{\sigma},s),L_{ex}(\Delta_{2p+1},s),\dots, L_{ex}(\Delta_{2p+q},s)$$ if $p\geq1$, and the $d-1$ first otherwise, we show by induction hypothesis, that $s_0$ is a pole of order $d-1$ of either $L_F^K(\Delta_{3} \times \dots \times \Delta_{t},s)$ if $p\geq1$, or of $L_F^K(\Delta_{1} \times \dots \times \Delta_{t-1},s)$ otherwise. In both situations, the pole must be exceptional according to Proposition \ref{distgen2}, and we deduce from Proposition \ref{L0}, that $s_0$ is a pole of order at least $d-1$ of $L_{(0)}(\pi,s)$.\\
Now let $\pi_i^{(k)}$ be an irreducible component of $\pi^{(k)}$ for a positive integer $k$. Then $\pi_i^{(k)}$ is equal to $\Delta_1^{(d_1)}\times\dots\times\Delta_t^{(d_t)}$ for non negative integers $d_i$ with sum equal to $k$. The function $L_{ex}(\Delta_1^{(d_1)}\times\dots\times\Delta_t^{(d_t)})$ divides by induction hypothesis the function $$\prod_{1\leq i<j \leq t} L(\Delta_i^{(d_i)} \times (\Delta_{j}^{(d_j)})^{\sigma},s)\prod_{1\leq k\leq t}L_F^K(\Delta_k^{(d_k)},s).$$ 
As each function $L(\Delta_i^{(d_i)} \times (\Delta_{j}^{(d_j)})^{\sigma},s)$ (resp. $L_F^K(\Delta_k^{(d_k)},s)$) divides $L(\Delta_i\times\Delta_{j}^{\sigma},s)$ (resp. $L_F^K(\Delta_k,s)$), and as $s_0$ occurs exactly in the functions $L(\Delta_{2i-1}\times\Delta_{2i}^{\sigma},s)$, for $i$ in ${1,\dots,p}$, and in the functions $L_F^K(\Delta_{2p+j},s)$, for $j$ in ${1,\dots,q}$, the order of $s_0$ in $L_{ex}(\Delta_1^{(d_1)}\times\dots\times\Delta_t^{(d_t)})$ is at most $d$.\\
Suppose it is $d$, then $s_0$ must occur in each $L(\Delta_{2i-1}^{(d_{2i-1})} \times (\Delta_{2i}^{(d_{2i})})^{\sigma},s)$ for $i$ in ${1,\dots,p}$, and in the functions $L(\Delta_{2p+j}^{(d_{2p+j})},s)$, for $j$ in ${1,\dots,q}$, and no other. But as $s_0$ is a pole of $L_{Rad(ex)}(\Delta_1^{(d_1)}\times\dots\times\Delta_t^{(d_t)})$, Proposition \ref{distgen2} implies that it is actually a pole of $L_{ex}(\Delta_{2i-1}^{(d_{2i-1})} \times (\Delta_{2i}^{(d_{2i})})^{\sigma},s)$ for $i$ in ${1,\dots,p}$, and  $L_{ex}(\Delta_{2p+j}^{(d_{2p+j})},s)$ for $j$ in ${1,\dots,q}$.\\
 As at least one of the $d_i$'s is positive, this would implie that $s_0$ is a pole of $L_{ex}(\Delta_{2i-1}^{(d_{2i-1})} \times (\Delta_{2i}^{(d_{2i})})^{\sigma},s)$ and $L_{ex}(\Delta_{2i-1} \times \Delta_{2i}^{\sigma} ,s)$ at the same time for some positive $d_{2i-1}$ or $d_{2i}$, or of $L_{ex}(\Delta_{2p+j}^{(d_{2p+j})},s)$ and $L_{ex}(\Delta_{2p+j},s)$ at the same time for some positive $d_{2p+j}$. But as $L_{ex}(\Delta_{2i-1}^{(d_{2i-1})} \times (\Delta_{2i}^{(d_{2i})})^{\sigma},s)$ divides $L_{(0)}(\Delta_{2i-1}\times \Delta_{2i}^{\sigma},s)$ and $L_{ex}(\Delta_{2p+j}^{(d_{2p+j})},s)$ divides $L_{(0)}(\Delta_{2p+j},s)$, this would imply that $s_0$ is a pole of order $2$ of $L(\Delta_{2i-1}\times\Delta_{2i}^{\sigma},s)$ or $L_F^K(\Delta_{2p+j},s)$, which contradicts the fact that these functions have simple poles.\\
Finally the order in any $L_{ex}(\pi_i^{(k)},s)$, for any positive $k$, and any irreducible component $\pi_i^{(k)}$ of $\pi^{(k)}$ is at most $d-1$. We deduce that $s_0$ occurs in $L_{(0)}(\pi)$ with order $d-1$ exactly, hence it occurs with order $d$ in $L_F^K(\pi)= L_{Rad(ex)}(\pi)L_{(0)}(\pi)$.

\end{proof}

\subsection{Deformation, derivatives and local factors}\label{subsectionder}

We again recall a few facts from \cite{CP}, Sections 3 and 4, about deformations of representations and behaviour of the derivatives.\\
Let $t$ be a positive integer, the map $u=(u_1,\dots,u_t)\mapsto q_K^u=(q_K^{u_1},\dots,q_K^{u_t})$ defines an isomorphism of varieties between $(\mathcal{D}_K)^t=(\mathbb{C}/\frac{2i \pi}{ln(q_K)\mathbb{Z}})^t$ and $(\mathbb{C}^*)^t$. We also denote by $D_F$ the variety $(\mathbb{C}/\frac{2i \pi}{ln(q_F)\mathbb{Z}})$ which the isomorphism $s\mapsto q_F^{-s}$ identifies to $(\mathbb{C}^*)^t$.\\

\begin{df-prop}\label{typeA} Let $(n_1,\dots,n_t)$ denote a partition of $n$ by positive integers, and let $\Delta_i$ be a quasi-square-integrable representation of $G_{n_i}(K)$ for $i$ between $1$ and $t$. As quasi-square-integrable representations are generic from Theorem \ref{classgen}, and according to \cite{R}, the representation $\pi=\Delta_1 \times\dots\times \Delta_t$ of the group $G_n(K)$ is of Whittaker type, though it is not necessarily irreducible. We say that such a representation $\pi$ is of parabolic type.\end{df-prop}
 
For $u \in \mathcal{D}_t$, and $\pi$ of parabolic type, we define the deformed representation $\pi_u=\Delta_1 |\ |_K^{u_1}\times\dots\times \Delta_t |\ |_K^{u_t}$.

\begin{df-prop}\label{gen2} Let $\pi$ be a parabolic type representation of the group $G_n(K)$, then outside a finite number of hyperplanes in $u$, the representation $\pi_u$ is in general position (see Definition \ref{gen}). If $\pi$ is fixed, we say that $u$ is in general position if the representation $\pi_u$ is.\end{df-prop}
\begin{proof} The fact that conditions 1. and 2. of Definition \ref{gen} are verified follows from the discussion before Proposition 3.5 of \cite{CP}. For condition 3., let $i$, $j$ and $k$ be three different integers between $1$ and $t$. It is enough to show that outside a finite number of hyperplanes in $u$, the functions $L(\Delta_i |\ |_K^{u_i} \times \Delta_j^{\sigma} |\ |_K^{u_j},s)$ and $L_F^K(\Delta_i |\ |_K^{u_i},s)$, and the functions $L(\Delta_i |\ |_K^{u_i} \times \Delta_j^{\sigma} |\ |_K^{u_j},s)$ and $L(\Delta_i |\ |_K^{u_i} \times \Delta_k^{\sigma} |\ |_K^{u_k},s)$ have no common pole. To do this we write $\Delta_i$ as $St_{m_i}(\rho_i)$ and $\Delta_j$ as $St_{m_i}(\rho_j)$. According to Theorem (8.2) of $\cite{JPS}$, we know that $L(\Delta_i |\ |_K^{u_i} \times \Delta_j |\ |_K^{u_j},s)$ is equal to $1$ unless $\rho_i$ and $\rho_j$ are supercuspidal representations of the same group $G_{m'}(K)$. So we suppose it is the case, and from the same proposition of $\cite{JPS}$, one has the equality $L(\Delta_i |\ |_K^{u_i} \times \Delta_j^{\sigma} |\ |_K^{u_j},s)=\prod_{k=0}^{m_i -1} L(\rho_i \times \rho_j^{\sigma}, s + u_i +u_j + (m_i-m_j)/2 + k)$. Suppose $L(\Delta_i |\ |_K^{u_i} \times \Delta_j^{\sigma} |\ |_K^{u_j},s)$ and $L_F^K(\Delta_i |\ |_K^{u_i},s)$ have a common pole at $s_0$, then $\rho_i^{\vee}=\rho_j^{\sigma}|\ |_K^{s_0+ u_i+u_j+ (m_i-m_j)/2 + l}$ for an integer $l$ between $0$ and $m_j-1$, and according to Proposition \ref{asaidiscrete} and Proposition \ref{Lcusp}, there exists $l'$ between $0$ and $m_i-1$ such that $\eta_{K/F}^{m_i -1-l'}|\ |_K^{(s_0+l')/2 +u_i}\rho_i$ is distinguished. We deduce from Proposition \ref{autodual} that $\rho_i^{\vee}$ must be equal to $|\ |_F^{s_0+l'+2u_i}\rho_i^{\sigma}$. This eventually implies that $\rho_j$ is equal to $\rho_i|\ |_K^{(m_j-m_i)/2 + l'-l +u_j- u_i}$. As $-(m_j+m_i)/2\leq (m_j-m_i)/2 + l'-l\leq (m_j+m_i)/2$, considering central characters, we deduce that $L(\Delta_i |\ |_K^{u_i} \times \Delta_j^{\sigma} |\ |_K^{u_j},s)$ and $L_F^K(\Delta_i |\ |_K^{u_i},s)$ have no common poles outside a finite number of hyperplanes in $(u_i,u_j)$.\\
We show in a similar way that $L(\Delta_i |\ |_K^{u_i} \times \Delta_j^{\sigma} |\ |_K^{u_j},s)$ and $L(\Delta_i |\ |_K^{u_i} \times \Delta_k^{\sigma} |\ |_K^{u_k},s)$ have no common pole outside a finite number of affine hyperplanes in $(u_i,u_j,u_k)$. This concludes the proof of the proposition.\end{proof}

Now we are going to prove that whenever the representation $\pi$ is of parabolic type, then the factor $\gamma_F^K(\pi_u,\psi,s)$ (see after Theorem \ref{Asaieq}) is a rational function of $q_K^{-{u}}$ and $q_F^{-{s}}$.\\

We denote by $Q$ the standard parabolic subgroup of the group $G_n(K)$ corresponding to the partition $(n_1,\dots,n_t)$, by $M$ the associated Levi subgroup, which is isomorphic to $G_{n_1}(K)\times\dots\times G_{n_t}(K)$, and $N_Q$ the unipotent radical of $Q$. Whenever $u$ is in $\mathcal{D_K}$, the representation $\pi_u$ is of Whittaker type from Proposition-Definition \ref{typeA}. For $m$ in $M$, we denote by $|m|_K^{u}$ the positive number $|det(m_1)|_K^{u_1}\dots|det(m_t)|_K^{u_t}$.\\
We first recall from \cite{CP} that $\pi_u$ can be realised in the space $\mathcal{F}_{\pi}$ of smooth funtions of restrictions to $K_{n}^K$ of functions $$f:G_n(K)\rightarrow W(\Delta_1,\psi)\otimes \dots \otimes W(\Delta_t,\psi)$$ satisfying $$f(hg,m)=\Delta_Q (m_h)^{-1/2}f(g,mm_h)$$ for $h$ in $Q$, $g$ in $G$, $m$ in $M$ and $m_h$ in $M$ such that $h=nm_h$ for $n$ in $N_Q$.\\
A function $f$ in $\mathcal{F}_{\pi}$ determines a function $f_u$ in the space of smooth funtions from $G_n(K)$ to $W(\Delta_1 |\ |_K^{u_1},\psi)\otimes \dots \otimes W(\Delta_t |\ |_K^{u_1},\psi)$, satisfying $f_u(hg,m)=\Delta(m_h)^{-1/2}f_u(g,mm_h)$ as before. To do this, we define for $g=n'm'k'$ with $n'$ in $N_Q$, $m'$ in $M$ and $k'$ in $K_n^K$, the value of the function $f_u$ to be $f_u(g,m)= |m'|_K^{u}\Delta(m')^{-1/2}f(k',mm')$. This indentifies the set of representations $\pi_u$ with $u$ in $\mathcal{D}_K$ with the trivial vector bundle $\mathcal{D}_K\times\mathcal{F}_{\pi}$, with $G_n(K)$ acting in each fibers with different action.\\
Considering $\pi_u$ as acting on $\mathcal{F}$, a Whittaker functional $$\Lambda_u(f)=\int_{N_f}(\pi_u(n)f)(w_Q,e)\psi^{-1}(n)dn$$ is defined in Section 3 of \cite{CP}, to which we refer for notations. For $f\in \mathcal{F}$, the Whittaker function $$W_{f,u}(g)=\Lambda_u(\pi_u(g)f)=\Lambda_0(\pi(g)f_u)=W_{f_u}(g)$$ is defined and describes $W(\pi_u,\psi)$ when $f$ describes $\mathcal{F}$.\\
 It is shown in \cite{CP} that for fixed g, the function $W_{f,u}(g)$ belongs to $\mathbb{C}[q_K^{\pm u_1},\dots,q_K^{\pm u_t}]$. 
As in \cite{CP}, we define $W_{\pi}^{(0)}$ the complex vector space generated by the functions $(u,g)\mapsto W_{{f_u}}(gg')$ for $g'\in G_n(K)$ and $f\in \mathcal{F}_{\pi}$. It is shown in \cite{CP} that the action of the group $G_n(K)$ on $W_{\pi}^{(0)}$ by right translation is a smooth representation, and we denote by $W_{\pi,(0)}$ the space of restrictions of functions of $W_{\pi}^{(0)}$ to $P_n$. As in \cite{CP}, we denote by $\mathcal{P}_0$ the vector subspace of $\mathbb{C}[q_K^{\pm{u}}]$ consisting of all Laurent polynomials of the form $u\mapsto W(u,I_n)$ for some $W\in W_{\pi}^{(0)}$. We now state Proposition 3.1 of \cite{CP} which we will use later.

\begin{prop}\label{incl}
Let $\pi$ be a representation of parabolic type, the complex vector space $W_{\pi,(0)}$ defined above contains the space $C_c^{\infty}(N_n(K)\backslash P_n(K), \mathcal{P}_0, \psi)$.
\end{prop}

In order to prove rationality properties of the Asai-gamma factor of $\pi_u$, we will show that if $f$ belongs to $\mathcal{F}_{\pi}$, and $\phi$ belongs to $C_c^{\infty}(F^n)$, the integral $I(W_{f_u},\phi,s)$ is a rational function in $q_F^{-u}$ and $q_F^{-s}$. To do this, we use as in \cite{CP} a theorem of Bernstein about solutions of a polynomial family of affine equations.\\

Let $V$ be a complex vector space of countable dimension. Let $R$ be an index set, and let $\Xi$ be a collection $\left\lbrace (x_r, c_r)| r\in R \right\rbrace$ with $x_r \in V$ and $c_r\in \mathbb{C}$. A linear form $\lambda$ in $V^*= Hom_{\mathbb{C}}(V,\mathbb{C})$ is said to be a solution of the system $\Xi$ if $\lambda(x_r)=c_r$ for all $r$ in $R$.\\ 
Let $\mathcal{D}$ be an irreducible algebraic variety over $\mathbb{C}$, and suppose that to each $d$, a system $\Xi_d =\left\lbrace (x_r(d), c_r(d))| r\in R \right\rbrace$ with the index set $R$ independant of $d$ in $\mathcal{D}$. We say that the family of systems $\left\lbrace \Xi_{d}, d \in \mathcal{D}\right\rbrace $ is polynomial if $x_r(d)$ and $c_r(d)$ belong respectively to $ \mathbb{C}[\mathcal{D}]\otimes_{\mathbb{C}} V$ and $\mathbb{C}[\mathcal{D}]$. Let $\mathcal{M}=\mathbb{C}(\mathcal{D})$ be the field of fractions of $\mathbb{C}[\mathcal{D}]$, we denote by $V_{\mathcal{M}}$ the space $\mathcal{M} \otimes_{\mathbb{C}} V$ and by $V_{\mathcal{M}}^*$ the space $Hom_{\mathcal{M}}(V_{\mathcal{M}},\mathcal{M})$. The theorem is the following (see Subsection 3.2. of \cite{CP}).

\begin{thm}{(Bernstein)}\label{Ber}
In the above situation, suppose that there exists a non-empty subset $\Omega \subset \mathcal{D}$ open in the usual complex topology of $\mathcal{D}$, such that for each $d$ in $\Omega$, the system $\Xi_d$ has a unique solution $\lambda_{d}$. Then the system $\Xi=\left\lbrace (x_r(d),c_r(d))|r\in R \right\rbrace$ over the field $\mathcal{M}=\mathbb{C}(\mathcal{D})$ has a unique solution $\lambda(d)$ in $V_{\mathcal{M}}^*$, and $\lambda(d)=\lambda_d$ is the unique solution of $\Xi_d$ on a subset of $\mathcal{D}$ which is the complement of a countable number of hypersurfaces. 
\end{thm}
\begin{proof} We reproduce the proof of Bernstein, which is extracted from a letter to Piatetski-Shapiro. It proceeds in three steps.

\begin{itemize}

 \item[Step 1:] The system $\Xi$ has a solution $f$. It is enough to check\\

(*) For each collection $\left\lbrace a_r \in \mathcal{M} | r \in R \right\rbrace $ in which all but a finite number of $a_r$ are $0$, then $\sum a_r x_r = 0$ implies $\sum a_r c_r = 0$.\\

Indeed, if we know (*), we can unambiguously define $f$ on the $\mathcal{M}$-linear span of $\left\lbrace x_r\right\rbrace$ and
then extend it arbitrarily to $V_{\mathcal{M}}$.\\
Suppose that (*) is not true, i.e., there exists a linear combination $\sum a_r x_r$ which is zero such
that $\sum a_r c_r \neq 0$. Multiplying ar by a polynomial we can assume that $a_r \in \mathbb{C}[D]$ and $\sum a_r \lambda_r \neq 0$. Then at some point $d \in \Omega$, the sum $\sum a_r \lambda_r (d)$ is different from $0$, i.e., $\Xi_d$ does not have a solution, which is a contradiction.\\

\item[Step 2:] The solution $f$ is unique.\\

Let $\xi_i$ be a $\mathbb{C}$-basis of $V$. then $f$ is defined by a countable number of functions $f(\xi_i) \in \mathcal{M}$. If
there are two solutions $f$, $f'$, then outside of a countable number of hypersurfaces $f(d)$ and
$f'(d)$ are defined and are sloutions of $\Xi_d$. Since on $\Omega$, the system $\Xi_d$ has a unique solution, we obtain $f(d) = f'(d)$ for $d$ outside of a countable number of hypersurfaces of $\Omega$. This implies that $f$ is equal to $f'$.\\

\item[Step 3:] Outside of a countable number of hypersurfaces, the system $\Xi_d$ has a unique solution which is equal to
$f(d)$.\\

Since $\Xi$ has a unique solution, each vector $\xi_i$ can be written as a finite linear combination
$x_r = \sum a_{i,r}\xi_i$. Put $D'= {\left\lbrace d | \ a_{i,r} \textrm{ and f are defined at d}\right\rbrace }$. Then for all $d\in \Omega$, the function $f(d)$ is the
unique solution of $\Xi_d$.

\end{itemize} \end{proof}

In order to apply this theorem, we will need the following proposition.

\begin{prop}
Let $\pi$ be a representation of parabolic type, there are $t$ affine linear forms $L_i$, for $i$ between $1$ and $t$, with $L_i$ depending on the variable $u_i$, such that if the $L_i(u_i)$'s and $s$ have positive real parts, the integral $I(W,\phi,s)$ is convergent for any $W$ in $W(\pi_u,\psi)$ and any $\phi$ in $C_c^{\infty}(F^n)$. 
\end{prop}
\begin{proof}
Let $W$ belong to $W(\pi_u,\psi)$ and $\phi$ belong to $C_c^{\infty}(F^n)$, because of the formula $$\int_{N_n(F)\backslash G_n(F)}\!\!\!\!\!\!\!\!\!\!\!\!\!\!\!\!\!\!\!\!\!\!\!W(g)\phi(\eta_n g)|det(g)|_F^s dg=\!\!\!\int_{K_{n,F}}\!\!\int_{N_n(F)\backslash P_n(F)}\!\!\!\!\!\!\!\!\!\!\!\!\!\!\!\!\!\!\!\! \!\!\!W(pk)|det(p)|_F^{s-1} dp \!\! \int_{F^*}\!\!\!\phi(\eta_nak)c_{\pi}(a)|a|_F^{ns}d^*a dk,$$ one deduces that the absolute convergence of the integral $$\int_{N_n(F)\backslash G_n(F)} W(g)\phi(\eta_n g)|det(g)|_F^s dg$$ is guaranteed by those of $\int_{N_n(F)\backslash P_n(F)} W'(p)|det(g)|_F^{s-1} dp$ for any $W'$ in $W(\pi_u,\psi)$, and of $\int_{F^*}\phi'(\eta_naI_n)|a|_F^{ns}d^*a$ for any $\phi'$ in $C_c^{\infty}(F^n)$.\\
 It is classical that the latter integrals converge as soon as $Re(s)>0$, so we focus on integrals of the form $\int_{N_n(F)\backslash P_n(F)} W'(p)|det(p)|_F^{s-1} dp$ for $W'$ in $W(\pi_u,\psi) $.\\
We now recall the following claim, which is proved in the lemma of Section 4 of \cite{F1}.

\begin{claim}
Let $\tau$ be a sub-$P_n(K)$-module of $C^{\infty}(N_n(K)\backslash P_n(K), \psi)$, such that for every $k$ between $0$ and $n$, the central exponents of $\tau^{(k)}$ are positive (i.e. the central characters of all the irreducible subquotients of $\tau^{(k)}$ have positive real parts), then whenever $W$ belongs to $\tau$, the integral $\int_{N_n(F)\backslash P_n(F)} W(p)dp$ is absolutely convergent.
\end{claim}

Applying this to our situation, and noting $e_{\pi}$ the maximal element of the set of central exponents of $\pi$ (see Section 7.2 of \cite{Ber}), we deduce that as soon as $u$ is such that $L_i(u)=u_i-e_{\pi} -1$ has positive real part for $i$ between $1$ and $t$, and as soon as $s$ has positive real part, the integral $\int_{N_n(F)\backslash P_n(F)} W'(p)|det(p)|_F^{s-1} dp$ converges for all $W'$ in $W(\pi_u,\psi)$.
\end{proof}

We now can prove the following:

\begin{prop}\label{rat}
Let $\pi$ be a representation of parabolic type, for every $W$ in $W_{\pi}^{(0)}$, and $\phi \in C_c^{\infty}(F^n)$, the function $I(W,\phi,s)$ belongs to $\mathbb{C}(q_F^{-u},q_F^{-s})$.
\end{prop}
\begin{proof}
In our situation, the underlying vector space is $V=\mathcal{F}_{\pi} \otimes C_c^{\infty}(F^n)$ and is of countable dimension because $\pi$ is admissible. The invariance property satisfied by the functional $I$,
 is \begin{equation}\label{inv} I(\pi_u(g)W_{f,u},\rho(g)\phi,s)=|det(g)|_F^{s}I(W_{f,u},\phi,s)\end{equation} for $f$ in $\mathcal{F}_{\pi}$ and $\phi$ in $C_c^{\infty}(F^n)$, and $g$ in $G_n(F)$.\\
From the proofs of Lemma 8 and of the unique proposition of \cite{F4}, it follows that out of the hyperplanes in $(u,s)$ defined by $c_{\rho_u}|z|^{j(s-1)}=1$, for $\rho_u$ in the irreducible components of $\pi_u^{(n-j)}$, and for $j>0$, and out of the hyperplane $c_{\pi_u}|z|^{ns}=1$, the space of solutions of equation \ref{inv} is of dimension at most one. If we take a basis of $(f_{\alpha})_{\alpha\in A}$ of $\mathcal{F}_{\pi}$, and a basis $(\phi_{\beta})_{\beta \in B}$ of $C_c^{\infty}(F^n)$, the polynomial family over the irreducible complex variety $\mathcal{D}=(\mathcal{D}_K)^t\times \mathcal{D}_F$ of systems $\Xi'_d$, for $d=(u,s)\in \mathcal{D}$ expressing the invariance of $I$ is given by:

$$\Xi'_d  = \left\lbrace \begin{array}{lc} ( \pi_u(g)\pi_u(g_i)W_{f_{\alpha},u}\otimes \rho(g)\phi_\beta - |det(g)|_F^{s} \pi_u(g_i) W_{f_{\alpha},u} \otimes \phi_\beta,0), \\ \alpha\in A, \beta \in B, g\in G_n(F), g_i \in G_n(K)   
\end{array} \right\rbrace $$
 
Now we define $\Omega$ to be the intersection of the three following subets of $\mathcal{D}$: 
\begin{itemize}
 \item the intersection of the complements of the hyperplanes such that $\pi_u$ is in general position on this intersection (see Proposition \ref{gen2}),
\item the intersection of the complements of the hyperplanes on which uniqueness up to scalar fails,
\item the intersection of the domains $\left\lbrace Re(L_i(u))>0 \right\rbrace $ and $\left\lbrace Re(s)>0 \right\rbrace $, on which $I(W_{f,u},\phi,s)$ is given by an absolutely convergent integral.
\end{itemize}

The functional $I$ is the unique solution up to scalars of the system $\Xi'$, in order to apply Theorem \ref{Ber}, we add for each $d\in \mathcal{D}$ a normalization equation $E_d$ depending polynomially on $d$. This is done as follows, we recall the following integration formula, which is valid for $W(u,g)\in W_{\pi}^{(0)}$ with $u\in \Omega$ and $\phi\in C_c^{\infty}(F^n)$:

$$\aligned I(W,\phi,u,s) & = \int_{N_n(F)\backslash G_n(F)} W(u,g)\phi(\eta_n g)|det(g)|_F^s dg\\
                                           & = \int_{K_n^F} \int_{N_n(F)\backslash P_n(F)} W(u, pk)|det(p)|_F^{s-1}dp \int_{F^*}
 \phi(\eta
 ak)c_{\pi}(a)|a|_F^{ns} d^*a dk. \endaligned$$

Now from Proposition \ref{incl}, if $F$ is a positive function in $C_c^{\infty}(N_n(K)\backslash P_n(K), \psi)$, we choose a $W$ in $W_{\pi}^{(0)}$ such that its restriction to $P_n(K)$ is of the form $W(u,p)=F(p) P(q_K^{\pm{u}})$ for some nonzero $P$ in $\mathcal{P}_0$. Let $K'$ be a sufficiently small subgroup of of $K_n^F$, such that $K'\cap P_n(K)$ stabilizes $W$, and let $\phi'$ be the characteristic function of $\eta_n K'$, then one has 
$$I(W,\phi',u,s)= r\int_{N_n(F)\backslash P_n(F)} F(p)|det(p)|_F^{s-1}dp P(q^{\pm{u}}),$$ for some positive constant $r$. Calling $c$ the constant $r\int_{N_n(F)\backslash P_n(F)} F(p)|det(p)|_F^{s-1}dp$, this latter equality becomes $I(W,\phi',u,s)=cP(q_K^{\pm{u}})$. Now as $W$ is in $W_{\pi}^{(0)}$, it can be expressed as a finite linear combination $W(g,u)= \sum_k \pi_u(g_i) W_{f_{\alpha},u}(g)$ for appropriate $g_{\alpha}\in G_n(K)$. Hence our polynomial family of normalization equations (which is actually independant of $s$) can be written $$E_{(u,s)}=\left\lbrace (\sum_{\alpha}  \pi_u(g_{\alpha}) W_{f_{\alpha},u} \otimes \phi', cP(q_K^{\pm{u}}) \right\rbrace.$$ 
We now call $\Xi$ the system given by $\Xi'$ and $E$, it satisfies the hypotheses of Theorem \ref{Ber}. We thus conclude that there is a functional $I'$ which is a solution of $\Xi$ such that $(u,s)\mapsto I'([u\mapsto W(u)],\phi,s)$ is a rational function of $q_F^{\pm u}$ and $q_F^{\pm s}$ for $W\in W_{\pi}^{(0)}$ and $\phi$ in $C_c^{\infty}(F^n)$. We also know that $I'([u\mapsto W(u)],\phi,s)$ is the unique solution of $\Xi'_{(u,s)}$ outside a countable number of hypersurfaces in $(u,s)$. In particular, on the intersection of $\Omega$ and of the complementary set of the removed hypersurfaces, the functionals $I$ and $I'$ agree. As $\mathcal{D}$ is locally compact, it is a Baire space, and this intersection contains . As they are rational functions of $q_F^{-s}$, they are equal for all $s$, and we conclude that $(u,s)\mapsto I(W_{f,u},\phi,s)$ belongs to $\mathbb{C}(q_F^{-u},q_F^{-s})$.\end{proof}

An immediate corollary is the following.

\begin{prop}
Let $\pi$ be a parabolic type representation of the group $G_n(K)$, then the factor $\gamma_F^K(\pi_u,s)$ belongs to $\mathbb{C}(q_F^{-u}, q_F^{- s})$. 
\end{prop}

We now prove a first step towards our main result.

\begin{prop}\label{div}
 Assuming Conjecture \ref{distgen}, let $\pi=\Delta_1 \times \dots \times \Delta_t$ be a parabolic type representation of the group $G_n(K)$, then the Euler factor $L_F^K(\pi,s)$ divides $\prod_{1\leq i<j \leq t} L(\Delta_i \times \Delta_{j}^{\sigma},s)\prod_{1\leq k\leq t}L_F^K(\Delta_k,s)$.
\end{prop}

\begin{proof}
We know from Theorem \ref{asgen} that $$L_F^K(\pi_u,s)=\prod_{1\leq i<j \leq m} L(\Delta_i \times \Delta_{j}^{\sigma},u_i +u_j+s)\prod_{1\leq k\leq m}L_F^K(\Delta_k,2 u_k +s)$$ when $u$ is in general position, hence for any $f\in \mathcal{F}_{\pi}$, and any $\phi$ in $C_c^{\infty}(F^n)$, the function $$\frac{I(W_{f,u},\phi,s)}{\prod_{1\leq i<j \leq m} L(\Delta_i \times \Delta_{j}^{\sigma},u_i +u_j+s)\prod_{1\leq k\leq m}L_F^K(\Delta_k,2u_k +s)}$$ has no poles when $u$ is in general position. But the removed hyperplanes defining general position don't depend on $s$, but only on $u$. As we know that for fixed $u$, the integral defining $I(W_{f,u},s)$ converges for $Re(s)$ large enough, we deduce that no polar locus of $I(W_{f,u},s)$ can lie in those removed hyperplanes. This implies that the rational function $I(W_{f,u},s)$ of $q_F^{-u}$ and $q_F^{-s}$ has no poles, hence belongs to $\mathbb{C}[q_F^{\pm u},q_F^{\pm s}]$. We now specialize at $u=0$, to get that $$\frac{I(W,\phi,s)}{\prod_{1\leq i<j \leq m} L(\Delta_i \times \Delta_{j}^{\sigma}, s)\prod_{1\leq k\leq m}L_F^K(\Delta_k, s)}$$  belongs to $\mathbb{C}[q_F^{\pm s}]$ for every $W$ in $W(\pi,\psi)$ and $\phi$ in $C_c^{\infty}(F^n)$. The conclusion follows.\end{proof}

\begin{prop} Assuming Conjecture \ref{distgen}, let $\pi=\Delta_1 \times \dots \times \Delta_t$ be a parabolic type representation of the group $G_n(K)$, the two rational functions $\gamma_F^K(\pi_u,\psi,s)$ and $\prod_{1 \leq i < j \leq t} \gamma(\Delta_i \times \Delta_j^{\sigma}, s+u_i +u_j) \prod_{k=1}^t \gamma_F^K(\Delta_k , s+ 2u_k)$ are equal up to a unit in $\mathbb{C}[q_F^{-u},q_F^{-s}]$.\end{prop} 
\begin{proof} We define the ratio $\epsilon_F^{K,0}(\pi_u,\psi,s)$ to be the function $$\frac{\gamma_F^K(\pi_u,\psi,s)\prod_{1\leq i<j \leq t} L(\Delta_i \times \Delta_{j}^{\sigma},s+u_i +u_j)\prod_{1\leq k\leq t}L_F^K(\Delta_k,s+2u_k)}{\prod_{1\leq i<j \leq t} L(\Delta_i^{\vee} \times (\Delta_{j}^{\sigma})^{\vee},1-s-u_i-u_j)\prod_{1\leq k\leq t}L_F^K(\Delta_k^{\vee},1-s-2u_k)}.$$
Hence for $u$ in general position, we have the equality $\epsilon_F^{K,0}(\pi_u,\psi,s)= \epsilon_F^K(\pi_u,\psi,s)$, which imples this second equality. Applying the functional equation \ref{Asaieq} twice, we deduce that for $u$ in general position, one has $\epsilon_F^{K,0}(\pi_u,\psi,s)\epsilon_F^{K,0}(\pi_u^{\vee},\psi^{-1}, 1-s)=1$. This equality of rational functions being true on the Zariski open subset of $u$ in general position, it is always true, and $\epsilon_F^{K,0}(\pi_u,\psi,s)$ is therefore a unit in $\mathbb{C}[q_F^{-u},q_F^{-s}]$. Combining the following equalities, $$\gamma_F^K(\pi_u,\psi,s)= $$ $$ \frac{\epsilon_F^{K,0}(\pi_u,\psi,s)\displaystyle \prod_{1\leq i<j \leq t} L(\Delta_i^{\vee} \times (\Delta_{j}^{\sigma})^{\vee},1-s-u_i-u_j)\displaystyle \prod_{1\leq k\leq t}L_F^K(\Delta_k^{\vee},1-s-2u_k)}{\displaystyle \prod_{1\leq i<j \leq t} L(\Delta_i \times \Delta_{j}^{\sigma},s+u_i +u_j)\displaystyle \prod_{1\leq k\leq t}L_F^K(\Delta_k,s+2u_k)},$$ 
 $$\gamma(\Delta_i \times \Delta_j^{\sigma}, s+u_i +u_j, \psi) = \frac{\epsilon (\Delta_i \times \Delta_j^{\sigma},s+u_i +u_j,\psi)L(\Delta_i^{\vee} \times (\Delta_{j}^{\sigma})^{\vee},1-s-u_i-u_j)}{L(\Delta_i \times \Delta_{j}^{\sigma},s+u_i +u_j)}$$ and $$\gamma_F^K(\Delta_k, s+2u_k, \psi) = \frac{\epsilon (\Delta_k, s+ 2u_k, \psi)L(\Delta_k^{\vee} ,1-s-2u_k)}{L(\Delta_k,s+2u_k)},$$ we conclude that $\gamma_F^K(\pi_u,\psi,s)$ and $\displaystyle \prod_{1 \leq i < j \leq t} \gamma(\Delta_i \times \Delta_j^{\sigma}, s+u_i +u_j) \prod_{k=1}^t \gamma_F^K(\Delta_k , s+ 2u_k)$ are equal up to a unit in $\mathbb{C}[q_F^{-u},q_F^{-s}]$.\end{proof} 

This proposition has an immediate corollary.

\begin{cor}\label{gammauptounit}
Assuming Conjecture \ref{distgen}, let $\pi=\Delta_1 \times \dots \times \Delta_t$ be a parabolic type representation of the group $G_n(K)$, the two rational functions $\gamma_F^K(\pi,\psi,s)$ and $\prod_{1 = i < j =t} \gamma(\Delta_i \times \Delta_j^{\sigma}, s) \prod_{k=1}^t \gamma_F^K(\Delta_k , s)$ are equal up to a unit in $\mathbb{C}[q_F^{-s}]$. 
\end{cor}

We will denote by $\sim$ the fact of being equal up to a unit in $\mathbb{C}[q_F^{-s}]$, hence we can write $\gamma_F^K(\pi,\psi,s)\sim\prod_{1 = i < j =t} \gamma(\Delta_i \times \Delta_j^{\sigma}, s) \prod_{k=1}^t \gamma_F^K(\Delta_k , s)$.

\begin{prop}\label{div2}
Suppose that $\pi_1$ and $\pi_2$ are representations of Whittaker type of $G_{n_1}(K)$ and $G_{n_2}(K)$, and $\pi=\pi_1\times \pi_2$ the corresponding representation of the group $G_n(K)$, for $n=n_1 +n_2$. Then the Euler factor $L_F^K(\pi_2,s)$ divides $L_F^K(\pi,s)$. 
\end{prop}

\begin{proof}We know from discussion after the formula (\ref{eqder}) that if $W$ belongs to $W(\pi,\psi)$, then the function $$I_{(n_1-1)}(W,s)= \int_{N_{n_2}(F)\backslash G_{n_2}(F)} W\left(\begin{array}{lll} h &  \\  & I_{n_1}\end{array}\right)  |det(h)|_F^{s-n_1} dh$$ belongs to $I(\pi)$. But from Proposition 9.1 of \cite{JPS}, for each couple $(W_2,\phi)$ in $W(\pi_2,\psi) \times C_c^{\infty}(K^{n_2})$, there exists $W$ in $W(\pi,\psi)$, such that $W\left(\begin{array}{lll} h &  \\  & I_{n_1}\end{array}\right)=W_2(h)\phi(\eta_{n_2}h)|det(h)|_K^{n_1 /2}$ for $h$ in $G_{n_2}(K)$. Hence each integral $$I(W_2,\phi,s)= \int_{N_{n_2}(F)\backslash G_{n_2}(F)}W_2(h)\phi(\eta_{n_2}h)|det(h)|_F^{s}dh $$ is equal to $I_{(n_1-1)}(W,s)$ for some $W$ in $W(\pi,\psi)$ because $|\ |_K^{n_1 /2}$ restricts as $|\ |_F^{n_1}$ to $F^*$. Hence every function $I(W_2,\phi,s)$ belongs to $I(\pi)$, which is $I(\pi_2)\subset I(\pi)$ and the conclusion follows.\end{proof}    

We are now able to show our main result.
\begin{thm}\label{asailang}
Let $(n_1,\dots,n_t)$ be a partition of the positive integer $n$, and for each $i$ between $1$ and $t$, let $\Delta_i$ be a square-integrable representation of $G_{n_i}(K)$. We take $\pi$ to be the Langlands type representation $\Delta_1|\ |^{u_1}\times \dots \times \Delta_1|\ |^{u_t}$ with the $u_i$'s real numbers ordered decreasingly. Assuming Conjecture \ref{distgen}, then one has the equality $$L_F^K(\pi,s)=\prod_{1\leq i<j \leq t} L(\Delta_i \times \Delta_{j}^{\sigma},s+u_i+u_j)\prod_{1\leq k\leq t}L_F^K(\Delta_k,s+2u_k).$$  
\end{thm}

\begin{proof}
 We denote by $L'_{W}(\pi,s)$ the function $$\displaystyle\prod_{1\leq i<j \leq t} \!\!\!\!L(\Delta_i \times \Delta_{j}^{\sigma},s+u_i+u_j)\!\!\displaystyle\prod_{1\leq k\leq t}L_F^K(\Delta_k,s+2u_k).$$ We prove the theorem by induction on $t$.\\
There is nothing to prove when $n=1$.\\  
$t-1\rightarrow t$: we assume that the result is true for any $k\leq t-1$. Let $\pi$ be the Langlands type representation $\Delta_1|\ |^{u_1}\times \dots \times \Delta_1|\ |^{u_t}$. Fom Proposition \ref{div}, there are two polynomials such $P$ and $\tilde{P}$ that $L_F^K(\pi,s)=P(q_F^{-s})L'_{W}(\pi,s)$ and $L_F^K(\pi^{\vee},s)=\tilde{P}(q_F^{-s})L'_{W}(\pi^{\vee},1-s)$. We want to show that $P$ is $1$.\\
 From Proposition \ref{gammauptounit}, we deduce $P(q_F^{-s})\sim \tilde{P}(q_F^{-s})$. Now we denote by $\pi_2$ the representation $\Delta_2|\ |^{u_2}\times \dots \times \Delta_t|\ |^{u_t}$, hence $\pi$ is equal to $\Delta_1|\ |^{u_1}\times \pi_2$. By induction hypothesis, we have the equality $L_F^K(\pi_2,s)=L'_{W}(\pi_2,s)$ and we know from Proposition \ref{div2} that there exists a polynomial $Q$ such that $L_F^K(\pi_2,s)= Q(q_F^{-s})L'_{W}(\pi,s)$.  We thus have $L_F^K(\pi_2,s)=Q(q_F^{-s})P(q_F^{-s})L'_{W}(\pi,s)$, and $P(q_F^{-s})$ divides $[\prod_{j\geq 2} L_F^K(\Delta_1\times \Delta_j^{\sigma},s+u_1+u_j)L_F^K(\Delta_1,s+2u_1)]^{-1}$. Corollary of Theorem 2.3 of \cite{CP} implies that $[\prod_{i\geq 2} L_F^K(\Delta_1\times \Delta_j^{\sigma},s+u_1+u_j)]^{-1}$ has its zeros in the set $\cup_{j=2}^t\left\lbrace s/ Re(s)\leq -u_1 -u_j  \right\rbrace$ and Corollary \ref{polediscrete}, that the function $[L_F^K(\Delta_1,s+2u_1)]^{-1}$ has its zeros in the set $\left\lbrace s/Re(s)\leq -2 u_1\right\rbrace$. Finally $P(q_F^{-s})$ has its zeros in the set $\left\lbrace s/Re(s)\leq -u_1- u_t\right\rbrace$.\\
Now we denote by $\pi_3$ the representation $\Delta_1|\ |^{u_1}\times \dots \times \Delta_{t-1}|\ |^{u_{t-1}}$, hence $\pi$ is equal to $\pi_3 \times \Delta_{t}|\ |^{u_{t}}$. We obtain in the same manner that $\tilde{P}(q_F^{-s})$ divides $[\prod_{j\leq t-1} L_F^K(\Delta_j^{\vee}\times (\Delta_t^{\sigma})^{\vee},1-s-u_t-u_j)L_F^K(\Delta_t,1-s-2u_t)]^{-1}$, hence its zeroes must have real part greater than or equal to $1-u_1-u_t$.\
This implies that $P(q_F^{-s})$ and $\tilde{P}(q_F^{-s})$ have no zero in common, but are equal up to a unit in $\mathbb{C}[q_F^{-s}]$, hence both are constant, and as $P(0)=\tilde{P}(0)=1$, both are equal to one.\end{proof}

As any generic representation, is isomorphic to a representation of Langlands type, we have the following corollary.

\begin{cor}\label{asaigen}
 Assuming Conjecture \ref{distgen}, if the representation $\pi$ is a generic representation of the group $G_n(K)$, then the functions $L_F^K(\pi,s)$ and $L_F^{K,W}(\pi,s)$ are equal.
\end{cor}
\begin{proof} We know that this equality is true for quasi-square-integrable representations. Let $\pi=\Delta_1\times \dots \times \Delta_t$ be a generic representation. Each quasi-square-integrable representation $\Delta_i$ is associated by Langlands correspondence to a representation $\rho_i$ of the Weil-Deligne group $W'_K$, and $\pi$ is thus associated with $\rho_1 \oplus \dots \oplus \rho_n$. From Lemma 7.1 of \cite{P}, one deduces that $L_F^{K,W}(\Delta_1\times \dots \times \Delta_n)$, which is by definition $L_F^K(\rho_1 \oplus \dots \oplus \rho_t)$, is equal to $\prod_{1\leq i<j \leq t} L(\rho_i \times \rho_{j}^{\sigma},s)\prod_{1\leq k\leq t}L_F^K(\rho_k,s)$, and the latter is known to be equal to $\prod_{1\leq i<j \leq t} L(\Delta_i \times \Delta_{j}^{\sigma},s)\prod_{1\leq k\leq t}L_F^K(\Delta_k,s)$. Conclusion then follows from Theorem \ref{asailang}.
\end{proof}

Finally, we know from the main theoem of \cite{M1}, that Conjecture \ref{distgen} is true for principal series representations of $G_n(K)$. The theorem is the following.

\begin{thm}
\label{distprincgn}
 
Let $\chi=(\chi_1,\dots,\chi_n)$ be a character of $T_n(K)$ with unlinked characters $\chi_i$ of $K^*$, the irreducible principal series representation $\pi(\chi)$ is distinguished if and only if
 there exists $r \leq n/2$, such that $ \chi _{i+1}^{\sigma} = {{\chi _i}}
 ^{-1} $ for $i=1,3,..,2r-1$, and that ${\chi _i} _{|F^*} =1$ for $ i > 2r$.

\end{thm}

Now as a consequence of this theorem, and of the results of this section, we have the following:

\begin{thm}\label{asaiprinc}
Let $\pi=\lambda_1 \times \dots \times \lambda_n$ be an irreducible principal series representation, where the $\lambda_i$'s are unlinked characters of $F^*$. Then we have the following equality of $L$-functions: $$L_F^K(\pi,s)= \prod_{1\leq i < j \leq n}L(\lambda_i \times \lambda_j^{\sigma},s)\prod_{k=1}^n L_F^K(\lambda_k,s) = L_F^{K,W}(\pi,s)$$
\end{thm}

\begin{proof}
One only has to show that the proof of Theorem \ref{asgen} adapts to the subclass of representations of the principal series. The only thing to notice to adapt the proof, is that in this case, the irreducible components of the derivatives of $\pi$ are still principal series representations. The validity of the result for $\pi$ not in general position, follows from the deformation arguments of Subsection \ref{subsectionder}. \end{proof}

Eventually, we proove the converse of Corollary \ref{asaigen}.

\begin{thm}\label{equiv}
The inductivity relation for Rankin-Selberg type Asai $L$-functions (or equivalently the equality of Asai $L$-functions of generic representations, and Asai $L$-functions of their Langlands parameter), is equivalent to the truth of Conjecture \ref{distgen}. 
\end{thm}
 
\begin{proof} There is one implication left.\\
Suppose that we have the inductivity relation for the Rankin-Selberg type Asai $L$-function. Let $\pi=\Delta_1\times \dots \times \Delta_t$ be a distinguished generic representation of $G_n(K)$. As $\pi$ is Galois-autodual, we can assume that there are three non negative integers $p$, $q$, and $r$, with $2r+p+q=t$, such that the $q$ last $\Delta_i$'s are $\eta_{K/F}$-distinguished, and non isomorphic to one another, the $q$ previous $\Delta_i$'s are distinguished, and non isomorphic to one another, and $\Delta_{i+1}^{\sigma} = \Delta_i^{\vee}$ for $i=1,3,..,2r-1$. We denote by $\pi_d$ the representation $\Delta_1\times \dots \times \Delta_{2r+p}$, and by $\pi_u$ the representation $\Delta_{2r+p+1}\times \dots \times \Delta_t$, hence one has $\pi=\pi_d\times \pi_u$. 
 Now by hypothesis, we have the equality $L_F^K(\pi)= L_F^K(\pi_u)L_F^K(\pi_d)L(\pi_u \times \pi_d^{\sigma})$. We are going to show that $0$ is an exceptional pole of $L_F^K(\pi)$ if and only if the representation $\pi_u$ is zero.\\
If $0$ is a pole of $L_F^K(\pi)$, its order in $L_F^K(\pi)$ is the same as its order in $L_F^K(\pi_d)$. Indeed, zero doesn't occur as a pole in $L_F^K(\pi_u)$, otherwise it would occur either as a pole in $L_F^K(\Delta_i)$, for $i>2r+p$, which is impossible as $\Delta_i$ cannot be distinguished and $\eta_{K/F}$-distinguished at the same time, or as a pole of $L_F^K(\Delta_j\times \Delta_k^\sigma)$, for $k>j>2r+p$, which cannot be because we supposed that the $\Delta_i$'s were distinct for $i>2r+p$ (we recall that $\eta_{K/F}$-distinguished irreducible representations are Galois-autodual). For similar reasons, zero doesn't occur as a pole in $L_F^K(\pi_d\times \pi_u^{\sigma})$.\\ 
But if $\pi_u$ is not zero, then $\pi_d$ is a representation of $G_r(K)$, for $r\leq n-1$. We are going to show that this implies that the function $L_F^K(\pi_d)$ divides $L_{(0)}(\pi)$, which will imply that the order of $0$ in the Asai $L$-function of $\pi$, is equal to its order in $L_{(0)}(\pi)$, so that zero won't be an exceptional pole of $L_F^K(\pi)$.\\
To do this, we recall three facts:
\begin{itemize}
 \item From Proposition 9.1 of \cite{JPS}, given $\phi$ in $C_c^\infty (F^r)$, and $W_d$ in $W(\pi_d,\psi)$, there is $W$ in $W(\pi,\psi)$, such that $W\left(\begin{array}{cc} g & 0 \\ 0 & I_{n-r} \end{array}\right)= W_d(g)\phi(\eta_r g)|det(g)|_F^{n-r}$ for $g$ in $G_r(F)$.

\item From Lemma 9.2 of \cite{JPS}, given $W$ in $W(\pi,\psi)$, there exist compact open subgroups $U_1$ and $U_2$ of $M_{n-1-r,r}(K)$ and $B_{n-1-r}^- (K)$ (the image of $B_{n-1-r}(K)$ under the transpose map) respectively, and $W^0$ in $W(\pi,\psi)$, such that $W^0\left(\begin{array}{ccc} g & 0 & 0 \\ x & b^- & 0 \\ 0 & 0 & 1 \end{array}\right)$ is equal to $W\left(\begin{array}{cc} g & 0 \\ 0 & I_{n-r} \end{array}\right)$ for $(g,x,b^-)$ in $G_r(F)\times U_1 \times U_2$, and zero for $(g,x,b^-)$ in the complementary set of $ G_r(F)\times U_1 \times U_2$ in $G_r(F)\times M_{n-1-r,r}(K)\times B_{n-1-r}^- (K)$. 

\item The following integration formula is valid for positive measurable functions on $N_{n-1}(F)\backslash G_{n-1}(F)$:
$$\int_{N_{n-1}(F)\backslash G_{n-1}(F)}\phi(g) dg=\!\!\!\!\!\!\!\!\!\!\!\!\!\underset{(h,x,b^-)\in N_{r}(F)\backslash G_{r}(F)\times M_{n-1-r,r}(F)\times B_{n-1-r}^- (F)}{\int}\!\!\!\!\! \phi \left(\begin{array}{ccc} h & 0 & 0 \\ x & b^- & 0 \\ 0 & 0 & 1 \end{array}\right)|det(h)|_F^{r-n+1} dx \ dh \ d_r(b^-)$$ for right invariant Haar measures $ dg$, $dx$, $dh$, and $d_r(b^-)$.  
\end{itemize}

From this we deduce if $W_d$ is in $W(\pi_d,\psi)$, and $\phi$ is in $C_c^\infty (F^r)$, one can choose $W$ in $W(\pi,\psi)$, and $W^0$ in $W(\pi,\psi)$ as above, such that $I(W_d,\phi,s)=\int_{N_{r}(F)\backslash G_{r}(F)}W_d(g)\phi(\eta_r g)|det(g)|_F^{s} dg=\int_{N_{r}(F)\backslash G_{r}(F)}W\left(\begin{array}{cc} g & 0 \\ 0 & I_{n-r} \end{array}\right)|det(g)|_F^{s +r-n} dg=I_{(n-r-1)}(W,s) $. This meromorphic function is in turn equal to $$\underset{(h,x,b^-)\in N_{r}(F)\backslash G_{r}(F)\times M_{n-1-r,r}(F)\times B_{n-1-r}^- (F)}{\int}\!\!\!\!\! W\left(\begin{array}{ccc} h & 0 & 0 \\ x & b^- & 0 \\ 0 & 0 & 1 \end{array}\right)|det(h)|_F^{s+r-n} dx \ dh \ d_r(b^-),$$ which is up to a positive constant equal to $I_{(0)}(W^0,s)= \int_{N_{n-1}(F)\backslash G_{n-1}(F)} W^0(g)|det(g)|_F^{s-1} dg$.\\
Finally, the Asai function $L_F^K(\pi_d)$ divides the Euler factor $L_{(0)}(\pi)$.\\
Eventually, if $\pi_u$ is equal to zero, it is a consequence of \cite{M3} that $\pi=\pi_d$ is distinguished.\end{proof}

\section*{Acknowledgements}

I would like to thank Professor James Cogdell for sending me the proof of Bernstein's theorem.

\end{document}